\documentclass[10pt,reqno]{amsart}

\usepackage[a4paper,left=26mm,right=26mm,top=28mm,bottom=28mm,marginpar=25mm]{geometry}

\usepackage{amsmath}
\usepackage{amssymb}
\usepackage{amsthm}
\usepackage[latin1]{inputenc}
\usepackage{eurosym}
\usepackage[dvips]{graphics}

\usepackage{graphicx}

\usepackage{epsfig}
\usepackage{hyperref}
\usepackage{dsfont}
\usepackage{color}
\usepackage[displaymath,mathlines]{lineno}

\allowdisplaybreaks

\usepackage{hyperref}

\usepackage{ifthen}

\makeindex

\renewcommand{\div}{\operatorname{div}}

\newcommand{\Rr}{{\mathbb{R}}}

\newcommand{\Nn}{{\mathbb{N}}}

\newcommand{\Tt}{{\mathbb{T}}}

\newcommand{\Ll}{{\mathcal{L}}}

\newcommand{\Aa}{{\mathcal{A}}}
\newcommand{\Bb}{{\mathcal{B}}}

\newcommand{\Mm}{{\mathcal{M}}}

\newcommand{\epsi}{\epsilon}
\def\d{{\rm d}}
\def\dx{{\rm d}x}
\def\dy{{\rm d}y}

\def\dt{{\rm d}t}
\def\ds{{\rm d}s}
\def\leq{\leqslant}
\def\geq{\geqslant}

\numberwithin{equation}{section}

\newtheoremstyle{thmlemcorr}{10pt}{10pt}{\itshape}{}{\bfseries}{.}{10pt}{{\thmname{#1}\thmnumber{
#2}\thmnote{ (#3)}}}
\newtheoremstyle{thmlemcorr*}{10pt}{10pt}{\itshape}{}{\bfseries}{.}\newline{{\thmname{#1}\thmnumber{
#2}\thmnote{ (#3)}}}
\newtheoremstyle{defi}{10pt}{10pt}{\itshape}{}{\bfseries}{.}{10pt}{{\thmname{#1}\thmnumber{
#2}\thmnote{ (#3)}}}
\newtheoremstyle{remexample}{10pt}{10pt}{}{}{\bfseries}{.}{10pt}{{\thmname{#1}\thmnumber{
#2}\thmnote{ (#3)}}}
\newtheoremstyle{ass}{10pt}{10pt}{}{}{\bfseries}{.}{10pt}{{\thmname{#1}\thmnumber{
A#2}\thmnote{ (#3)}}}

\theoremstyle{thmlemcorr}

\newtheorem{theorem}{Theorem}
\numberwithin{theorem}{section}

\theoremstyle{thmlemcorr*}
\newtheorem{theorem*}{Theorem}
\newtheorem{lemma*}[theorem]{Lemma}
\newtheorem{corollary*}[theorem]{Corollary}
\newtheorem{proposition*}[theorem]{Proposition}
\newtheorem{problem*}[theorem]{Problem}
\newtheorem{conjecture*}[theorem]{Conjecture}

\theoremstyle{defi}
\newtheorem{definition}[theorem]{Definition}
\newtheorem{hyp}{Assumption}

\newtheorem{problem}{Problem}

\theoremstyle{remexample}
\newtheorem{remark}[theorem]{Remark}

\newtheorem{teo}[theorem]{Theorem}
\newtheorem{lem}[theorem]{Lemma}
\newtheorem{pro}[theorem]{Proposition}
\newtheorem{cor}[theorem]{Corollary}

\newtheorem*{notation}{Notation}

\theoremstyle{ass}

\begin{document}

\title[Time-dependent monotone Mean-Field Games]{Existence of weak solutions \\to time-dependent Mean-Field Games}

\author{Rita Ferreira}
\address[R. Ferreira]{
        King Abdullah University of Science and Technology (KAUST), CEMSE Division, Thuwal 23955-6900, Saudi Arabia.}
\email{rita.ferreira@kaust.edu.sa}
\author{Diogo Gomes}
\address[D. Gomes]{
        King Abdullah University of Science and Technology (KAUST), CEMSE Division, Thuwal 23955-6900, Saudi Arabia.}
\email{diogo.gomes@kaust.edu.sa}
\author{Teruo Tada}
\address[T. Tada]{
       King Abdullah University of Science and Technology (KAUST), CEMSE Division, Thuwal 23955-6900, Saudi Arabia.}
\email{teruo.tada@kaust.edu.sa}



\thanks{R. Ferreira, D. Gomes, and T. Tada were partially supported by baseline and start-up funds from King Abdullah University of Science and Technology (KAUST) OSR-CRG2017-3452.}
\date{\today}

\begin{abstract}
Here, we establish the existence of weak solutions to a wide class of time-dependent monotone mean-field games (MFGs).
These MFGs are given as a system of degenerate parabolic equations with initial and terminal conditions. To construct these solutions, we consider a high-order elliptic regularization in space-time. Then, using Schaefer's fixed-point theorem, we obtain the existence and uniqueness for this regularized problem. Using Minty's method, we prove the existence of a weak solution to the original MFG. Finally, the paper ends with a discussion on congestion problems and density constrained MFGs.
\end{abstract}

\maketitle

\section{Introduction}\label{intro}
To model the behavior of large populations of competing rational agents, 
Lasry and Lions, in \cite{ll1}, \cite{ll2}, and \cite{ll3}, and, independently, Caines, Huang, and Malham\'{e}, in \cite{Caines2}, \cite{Caines1}, introduced a class of problems now called
mean-field games (MFGs). In these games, agents
are indistinguishable and
seek to minimize an individual cost that depends on the
statistical distribution of the population. 

Here, we consider the following time-dependent MFG with space-periodic boundary conditions.
\begin{problem}\label{OP}
Let $T>0$ and \(d\in\Nn\), and define  $\Omega_T=(0,T)\times\Tt^d$,
where \(\Tt^d\) is the \(d\)-dimensional torus. Let $X(\Omega_T)$ and $\Mm_{ac}(\Omega_T)$ denote, respectively, the space of measurable functions on $\Omega_T$ and
the space of positive measures
on $\Omega_T$ that are finite and absolutely continuous with respect to the Lebesgue
measure. Fix $\gamma>1$ and $k\in \Nn$  such that $2k >\frac{d+1}{2} + 3$.   Assume that $a_{ij}\in C^2(\Tt^d)$ for $1\leq i,j\leq d$, $V \in L^{\infty}(\Omega_T)\cap C(\Omega_T)$, $g\in C^1(\Rr_0^+\times \Rr)$, $h:\Mm_{ac}(\Omega_T)\to X(\Omega_T)$ is a (possibly nonlinear) operator, 
$m_0$, $u_T\in C^{4k}(\Tt^d)$, and $H\in C^2(\Tt^d\times\Rr^d)$ are such that $A(x)=(a_{ij}(x))$ is a symmetric positive semi-definite matrix for each  $x\in \Tt^d$, $m_0>0$, $\int_{\Tt^d}m_0(x)\, \dx=1$, and $m\mapsto g(m,h(\boldsymbol{m}))$ is monotone with respect to the $L^2$-inner product.
Find $(m,u)\in L^1(\Omega_T)\times L^\gamma((0,T);W^{1,\gamma}(\Tt^d))$ satisfying $m\geq0$ and
\begin{equation}
\label{Highd1o}
\begin{aligned}
\begin{cases}
u_t + \sum_{i,j=1}^da_{ij}(x)u_{x_ix_j} - H(x,Du) +g(m,h(\boldsymbol{m}))
+ V(t,x)=0 & \text{in } \Omega_T, \\[1mm]
%
m_t - \sum_{i,j=1}^d(a_{ij}(x)m)_{x_ix_j}-\div\big(mD_pH(x,Du)\big)
= 0 & \text{in } \Omega_T, \\[1mm]
\displaystyle m(0,x)=m_{0}(x),u(T,x)=u_T(x) & \text{on } \Tt^d.
\end{cases}
\end{aligned}
\end{equation}
\end{problem}
The first equation in \eqref{Highd1o} is a Hamilton--Jacobi equation
that determines the value function, \(u\), for a typical agent. The second equation, 
the Fokker--Planck equation, gives the evolution of the distribution
of the agents, $m$. The initial-terminal conditions for $u$ and $m$
in \eqref{Highd1o} model the case where the initial distribution, $m_0$, of the agents is known, and agents seek to optimize a control problem with terminal cost $u_T$. 
 In Problem~\ref{OP}, and in the sequel, the elements of $\Mm_{ac}(\Omega_T)$ are denoted with a boldface font and  their
 densities with the same non-boldface letter. Hence, we define $\boldsymbol{m}$ as $\boldsymbol{m}:=m \Ll^{d+1}\lfloor_{\Omega_T}$, where $\Ll^{d+1}$ is the $(d+1)$-dimensional Lebesgue measure. Moreover, we write \(g(m,h(\boldsymbol{m}))(x,t)\)  in place
of \(g(m(x,t),h(\boldsymbol{m})(x,t))\). 
The coupling, $g$, between the Hamilton--Jacobi equation
and the Fokker--Planck equation comprises  
a ``local" dependence, via the dependence on \(m\),  and a non-local 
dependence, via the operator 
 $h$ evaluated at \(\boldsymbol{m}\).
This coupling describes the interactions
between agents and the mean-field.
Because agents solve a control problem, the Hamiltonian, 
 $H=H(x,p)$, is convex in $p$ (Assumption~\ref{Hconv} below);
moreover, the associated Lagrangian, \(L=L(x,v)=\sup_p \{ -p\cdot v-H(x,p)\}\), gives the agent's cost to move at speed $v$. 
The matrix $A=(a_{ij})$ yields the diffusion for agents. 
Finally, the potential, $V$, determines the spatial and time preferences of each agent.

In recent years, MFGs have been studied intensively.
Thanks to the regularizing properties of the Laplacian, both elliptic and parabolic MFGs are now well-understood.
For example, the existence of solutions to second-order time-dependent MFGs without congestion was examined in \cite{cgbt}, \cite{ciranttonon}, \cite{Gomes2015b}, \cite{Gomes2016c}, \cite{GPM2}, \cite{GPM3}, and \cite{GP2}. 
Time-dependent cases with congestion were investigated in \cite{GVrt2}, \cite{porretta}, and \cite{porretta2}. The time-dependent MFG with nonlocal coupling is studied in \cite{cllp13}.

As we explain next, 
various time-dependent first-order MFGs models were examined 
by several authors. 
 Each of the models presents distinct difficulties
 that are addressed with methods that rely heavily on the structure of the
particular problem. 
In \cite{CaCaCa2018}, the authors assume that the Lagrangian is quadratic.
In \cite{Cd1}, the existence and uniqueness of solution was examined when $H(\cdot,p)$ is quadratic-like and $g$ is Lipschitz continuous and $g(\cdot,m)$ is bounded for the $C^2$-norm.
In \cite{Cd2} and \cite{GraCard}, the authors suppose that the growth of $g(\cdot,m)$ and the growth of $H(\cdot, p)$ are of the form $m^{q-1}$ and $|p|^r$ respectively, where $q>1$ and $r>\max\{d(q-1),1\}$. 
In \cite{San16}, the density $m$ satisfies $0\leq m(t,x)\leq \overline m$ for all $(t,x)$, where $\overline m$ is given and $\overline m>1$.
In \cite{graber2017sobolev}, the authors assume only that the growth of $H$ is greater than $|p|^r$  and $g$ satisfies $\frac{1}{C} |m|^{q-1}\leq g(x,m)\leq C|m|^{q-1}+C$ for all $m\geq1$ and $x\in \Tt^d$, where $r>1$, $q>1$, and $C$ is a constant.
In \cite{Mayorga2018}, the existence and uniqueness of short-time solution to first-order time-dependent MFGs is addressed when $H$ is only of class $C^3$ in space; that is, we do not need convexity nor coercivity for $H$.
In \cite{OrPoSa2018}, the growth of $g$ is $m^{q-1}$, where $q>1$, and $H$ satisfies $\frac{1}{2c_H}|p|^2 -\gamma_H^-(x)\leq H(x,p)\leq \frac{c_H}{2}|p|^2 + \gamma^+_H(x)$, where $\gamma_H^-(x)=c_1 (1+|x|)$ and $\gamma_H^+(x)=c_2 (1+|x|^2)$ with $c_1, c_2>0$.
In \cite{PrSa2016}, the authors focus on $H(p)=\frac{1}{2}|p|^2$ and assume that $G'(m)=g(m)$ and $G$ is superlinear and strictly convex.
However, the degenerate case is less well studied. In \cite{FG2}, the authors constructed a monotonicity method to solve the stationary MFGs with degenerate terms and non-local terms. 
This method is one of the few tools that can be applied to a diverse class of MFGs - local and non-local, with or without congestion, first or second-order (including possibly degenerate problems).  
Therefore, in this paper, we extend the monotonicity method to time-dependent MFGs with degenerate terms and non-local terms to construct weak solutions to Problem~\ref{OP} for any terminal time (see Section~\ref{PfMT2}).

Problem~\ref{OP} encompasses multiple difficulties: the second-order terms may be degenerate,{\tiny } and the coupling can include both local and non-local terms.
Using monotonicity methods, we prove the existence of weak solutions under a standard set of assumptions discussed in Section~\ref{Ass}.

Let $m_0$ and $u_T$ be as in Problem~\ref{OP}.
Throughout this paper, $\Aa$, $\widehat\Aa$, $\Aa_0$, $\Bb,$ and $\Bb_0$ are the sets 
\begin{align}
\Aa&:= \Big\{m \in H^{2k}(\Omega_T)\ |\ m(0,x)=m_0(x),\ m\geq 0\Big\} \label{DAa},\\
\widehat\Aa&:=
\left\{
m\in \Aa\ \Big|\ \int_{\Tt^d}m(t,x)\,\dx=1\ \mbox{for a.e. }t\in[0,T]
\right\} \label{DAawh},\\
\Aa_0&:=\Big\{m \in H^{2k}(\Omega_T)\ |\ m(0,x)=0,\ m+m_0\geq 0\Big\}\label{DAa0},\\
\Bb&:=\Big\{w\in H^{2k}(\Omega_T)\ |\ w(T,x)=u_T(x) \Big\}, \label{DBb}
\\
\Bb_0&:=\Big\{w\in H^{2k}(\Omega_T)\ |\ w(T,x)=0 \Big\} \label{DBb0}.
\end{align}
%
\begin{definition}\label{DOPWS1}
A weak solution to Problem \ref{OP} is a pair $(m,\tilde u) \in L^1(\Omega_T) \times L^\gamma((0,T);W^{1,\gamma}(\Tt^d))$ satisfying
\begin{flalign}
({\rm D1})\enspace  &
m \geq 0 \ \mbox{a.e.~in }\Omega_T,  &\nonumber \\ 
({\rm D2})\enspace & \, \left\langle
   F
  \begin{bmatrix}
      \eta \\
      v
  \end{bmatrix},
  \begin{bmatrix}
      \eta  \\
      v
  \end{bmatrix}
  -
  \begin{bmatrix}
      m  \\
      \tilde u
  \end{bmatrix}
\right\rangle \geq 0
\quad \mbox{for all}\ (\eta, v)\in \widehat\Aa \times \Bb,&\nonumber
\end{flalign}
where $F:H^{2k}(\Omega_T; \Rr_0^+) \times H^{2k}(\Omega_T) \to (L^1(\Omega_T)\times L^1(\Omega_T))^\ast$ 
is given by
\begin{equation}\label{DRF2}
\begin{aligned}
\left\langle
   F
  \begin{bmatrix}
      \eta \\
      v
  \end{bmatrix},
  \begin{bmatrix}
      w_1 \\
      w_2
  \end{bmatrix}
\right\rangle
:=&\,
\int_0^T\!\!\!\int_{\Tt^d}\bigg( v_t + \sum_{i,j=1}^da_{ij}(x)v_{x_ix_j} -H(x,Dv) + g(\eta, h(\boldsymbol{\eta})) + V(t,x)\bigg)w_1 \, \dx\dt\\ 
&\quad+\int_0^T\!\!\!\int_{\Tt^d}\bigg(
\eta_t - \sum_{i,j=1}^d(a_{ij}(x)\eta)_{x_ix_j} - \div\big(\eta D_pH(x,Dv)\big) \bigg)w_2\, \dx\dt .
\end{aligned}
\end{equation}
\end{definition}

Here, we establish the existence of weak solutions to Problem~\ref{OP} as stated in the following theorem. 
\begin{teo}\label{TOP}
Consider Problem \ref{OP} and suppose 
that Assumptions \ref{Hconv}--\ref{Hmono} hold. Then, there exists a weak solution $(m,\tilde u)\in L^1(\Omega_T) \times L^\gamma((0,T);W^{1,\gamma}(\Tt^d))$ to Problem \ref{OP} in the sense of Definition~\ref{DOPWS1}.
\end{teo}

To prove Theorem~\ref{TOP}, we introduce a regularized problem, Problem~\ref{MP} below. This regularized problem is obtained from Problem~\ref{OP} by adding a high-order elliptic regularization on $[0,T]\times \Tt^d$. 
Due to this regularization, and using Schaefer's fixed-point theorem, we can prove that there exists a unique weak solution to Problem~\ref{MP} (see Section~\ref{PfMT1}). Then, we consider the limit with respect to the regularization parameter, $\epsilon\to 0$, to obtain a weak solution to Problem~\ref{OP} (see Section~\ref{PfMT2}).

Before stating Problem~\ref{MP}, we introduce some notation regarding
partial derivatives used throughout this manuscript.

\begin{notation}
Let  $\Omega_T=(0,T)\times\Tt^d
$ be as in Problem~\ref{OP},  let \((t,x) = (t, x_1, ..., x_d)\) denote an arbitrary
point in \(\Omega_T\), and let \(\upsilon :\Omega_T\to\Rr\) be sufficiently regular so that the following
partial derivatives make sense, at least in a weak sense. Fix
\(i,\,j\in\{1,...d\}\), \(\ell\in\Nn\), \(\alpha=(\alpha_1, ..., \alpha_d)\in\Nn^d_0\),
and \(\beta=(\beta_0,\beta_1,...,\beta_d)\in \Nn^{d+1}_0\); in
this manuscript, we often write  
\begin{itemize}
\item[\(\triangleright\)]   \(\upsilon_t\)  in place of
\(\frac{\partial\upsilon}{\partial t}\) and \(\partial_t^\ell \upsilon\)
in place of
\(\frac{\partial^\ell \upsilon}{\partial t^\ell }\);

\item[\(\triangleright\)]  \(\upsilon_{x_i
x_j}\) in place of
\(\frac{\partial^2\upsilon}{\partial x_i \partial x_j}\)  and
\(\partial_x^\alpha \upsilon\)
in place of
\(\frac{\partial^{|\alpha|} \upsilon}{\partial x_1^{\alpha_1}
... \partial x_d^{\alpha_d} }\);
\item[\(\triangleright\)]  \(\partial_{t,x}^\beta \upsilon\)
in place of
\(\frac{\partial^{|\beta|} \upsilon}{\partial t^{\beta_0}\partial x_1^{\beta_1}
... \partial x_d^{\beta_d} }\).
\end{itemize}
Also, to simplify the notation, we often omit the domain of the
multi-index of a partial derivative. For instance, if we write
\(\partial_x^\alpha \upsilon\), we implicitly assume that \(\alpha\in
\Nn^d_0\), while if we write  
\(\partial_{t,x}^\beta \upsilon\), we implicitly assume that \(\beta\in
\Nn^{d+1}_0\). Similarly, we write \(\sum_{|\alpha|= \ell}
\partial_x^\alpha \upsilon\) in place of  \(\sum_{\alpha\in\Nn^d_0,|\alpha|= \ell}
\partial_x^\alpha \upsilon\), for instance.

To simplify the notation, we write \(Dv\) and \(\div v \) in place of \(D_x v\)  and \(\div_x v \) to denote the gradient and divergence of \(v\), respectively,  with respect to \(x\in \Tt^d\). 
Finally,  \(\Delta_{x,t}\) denotes the Laplacian operator
with respect to the variables $(t,x)$; that is, \(\Delta_{x,t}\upsilon = \frac{\partial^2 \upsilon}{\partial t^2 } + \sum_{i=1}^d \frac{\partial^2 \upsilon}{\partial x_i^2 }\).

 \end{notation}


\begin{problem}\label{MP}

Let $T>0$ and \(d\in\Nn\), and define  $\Omega_T=(0,T)\times\Tt^d$.
 Let $X(\Omega_T)$ and $\Mm_{ac}(\Omega_T)$
be the spaces introduced in Problem~\ref{OP}. Fix $\epsilon\in(0,1)$ and $k\in \Nn$  such that $2k >\frac{d+1}{2}
+ 3$.   Assume that $a_{ij}\in C^2(\Tt^d)$ for $1\leq i,j\leq d$, $V\in L^\infty(\Omega_T)\cap C(\Omega_T)$, $\sigma$, $\xi \in C^{4k}(\overline{\Omega}_T)$, $g\in C^1(\Rr_0^+\times \Rr)$, $h:\Mm_{ac}(\Omega_T)\to X(\Omega_T)$ is a (possibly nonlinear) operator, 
$m_0$, $u_T\in C^{4k}(\Tt^d)$, and $H\in C^2(\Tt^d\times\Rr^d)$ are such that, for $x\in \Tt^d$, $A(x)=(a_{ij}(x))$ is a symmetric positive semi-definite matrix, $\sigma\geq 0$, $m_0> 0$, $\int_{\Tt^d}m_0(x)\,\dx=1$, and $m\mapsto g(m,h(\boldsymbol{m}))$ is monotone with respect to the $L^2$-inner product.  Find $(m, u) \in H^{2k}(\Omega_T) \times H^{2k}(\Omega_T)$ satisfying the condition $m\geq 0$, the system
\begin{equation}\label{HighdRP}
\begin{aligned}
\begin{cases}
u_t + \sum_{i,j=1}^d a_{ij}(x)u_{x_ix_j} - H(x,Du) +g(m,h(\boldsymbol{m}))
+ V(t,x)  +\epsi\sum_{|\beta|\in\{0,2k\}} \partial^{2\beta}_{t,x} (m+\sigma) = 0 & \text{in } \Omega_T\\[1mm]
%
 m_t - \sum_{i,j=1}^d\big(a_{ij}(x)(m+\sigma)\big)_{x_ix_j}-
\div\big((m+\sigma)D_pH(x,Du)\big) +\epsilon\sum_{|\beta|\in\{0,2k\}} \partial^{2\beta}_{t,x} (u+\xi)
  = 0& \text{in } \Omega_T\\[1mm]
m(0,\cdot)=m_{0}(\cdot),\enspace u(T,\cdot)=u_T(\cdot) & \text{on } \Tt^d
\end{cases}
\end{aligned}
\end{equation}
and, for each $i\in\Nn$ with $2\leq i\leq 2k$,
 the boundary  conditions 
\begin{equation}\label{BCum}
\begin{aligned}
&\textstyle 
\sum_{j=1}^{2k} \partial^{2j-1}_t (L_j m) = 0 \enspace \text{on } \{T\}\times\Tt^d\quad \text{and}\quad \sum_{j=i}^{2k} \partial^{2j-i}_t (L_j m) = 0 \enspace \text{on
} \{0,T\}\times\Tt^d,  \\[1mm]
&\textstyle \sum_{j=1}^{2k} \partial^{2j-1}_t (L_j u) = 0 \enspace \text{on
} \{0\}\times\Tt^d\quad \text{and}\quad \sum_{j=i}^{2k} \partial^{2j-i}_t
(L_j u) = 0 \enspace \text{on
} \{0,T\}\times\Tt^d,
\end{aligned}
\end{equation}
where \(L_j := \sum_{|\alpha|=2k-j} \partial^{2\alpha}_x\).
\end{problem}
In the preceding problem, $\sigma$ and $\xi$ are used to transform the boundary conditions for $m$ and $u$ into homogeneous boundary conditions (see Section~\ref{PfMT1}).
Furthermore, the boundary conditions \eqref{BCum} at the initial and terminal time were selected to preserve the monotonicity of Problem~\ref{OP} in the sense of Assumption~\ref{Hmono} (also see Remark~\ref{rmk:intparts} below). Generally, monotonicity may not hold with arbitrary boundary conditions. 
Also, the boundary conditions above are natural for our construction of solutions that uses a variational approach (see Section~\ref{VP}).
Because of the high-order terms in Problem~\ref{MP}, we do not expect that the maximum principle holds for the second equation. Hence, there may not be classical solutions with $m\geq 0$. Thus, as in \cite{FGT1}, we introduce a notion of weak solutions to Problem~\ref{MP}. This definition is related to the ones in \cite{cgbt} and \cite{FG2}, where $u$ is only a subsolution to the Hamilton--Jacobi equation. 
To construct weak solutions, we introduce two auxiliary problems: a variational problem and a problem given by a bilinear form, which correspond to the first and the second equations in Problem~\ref{MP}, respectively (see Sections~\ref{VP} and \ref{BF}).

%
\begin{definition}\label{DMPWS1}
A weak solution to Problem \ref{MP} is a pair $(m,u) \in H^{2k}(\Omega_T) \times H^{2k}(\Omega_T)$ satisfying, for all $w\in \Aa$ and $v \in \Bb_0$,
\begin{flalign*}
 ({\rm E1})\enspace & \,  (m,u)\in\Aa\times\Bb, &\\
 \begin{split}  ({\rm E2})\enspace &  \int_0^T\!\!\!\int_{\Tt^d} \bigg(   u_t + \sum_{i,j=1}^da_{ij}(x)u_{x_ix_j} - H(x,Du) +g(m,h(\boldsymbol{m})) + V(t,x) \bigg)(w-m) \, \dx\dt \\ 
&\quad +\int_0^T\!\!\!\int_{\Tt^d} \bigg[ \epsilon\bigg (m+\sigma+\sum_{|\beta|=2k} \partial^{2\beta}_{t,x} \sigma\bigg)(w-m)
+\epsilon \sum_{|\beta|=2k}\partial_{t,x}^\beta
m\big(\partial_{t,x}^\beta w-\partial_{t,x}^\beta m\big)\bigg]\, \dx\dt \geq 0,\end{split} &\\
\begin{split}  ({\rm E3})\enspace&   \int_0^T\!\!\!\int_{\Tt^d} \bigg(
m_t - \sum_{i,j=1}^d\big(a_{ij}(x)(m+\sigma)\big)_{x_ix_j} - \div\big((m+\sigma) D_p H(x,Du)\big) \bigg)v\,\dx\dt  \\
&\quad+\int_0^T\!\!\!\int_{\Tt^d} \bigg[\epsilon\bigg (u+\xi+\sum_{|\beta|=2k}
\partial^{2\beta}_{t,x}  \xi\bigg)v
+ \epsilon \sum_{ |\beta|= 2k}\partial_{t,x}^\beta u\partial_{t,x}^\beta v
 \bigg]\, \dx\dt=0.
\end{split} 
\end{flalign*}
\end{definition}

\begin{remark}\label{rmk:intparts} Assume that \((m,u)\) is a classical solution to Problem~\ref{MP}, and let \(v\in H^{2k}(\Omega_T)\) with \(v(T,\cdot) = 0\) on \(\Tt^d\). Then, integrating by parts and using  \eqref{BCum} with \(i=2k\), we obtain
\begin{equation*}
\begin{aligned}
&\int_{\Omega_T} \sum_{ |\beta|= 2k}\partial_{t,x}^\beta u\partial_{t,x}^\beta
v\,\dx\dt =\sum_{\ell=0}^{2k} \sum_{ |\alpha|= 2k-\ell}\int_{\Omega_T} \partial_{t}^\ell\partial_{x}^\alpha u\,\partial_{t}^\ell\partial_{x}^\alpha
v\,\dx\dt  = \sum_{\ell=0}^{2k}  (-1)^{2k-\ell} \int_{\Omega_T} \partial_{t}^\ell L_\ell u\,\partial_{t}^\ell v\,\dx\dt \\&\quad= \sum_{\ell =0}^{2k-1}  (-1)^{2k-\ell } \int_{\Omega_T} \partial_{t}^\ell L_\ell u\,\partial_{t}^\ell v\,\dx\dt - \int_{\Omega_T} \partial^{2k+1}_t L_{2k} u\, \partial^{2k-1}_t v \,dx\dt\\
&\quad = \sum_{\ell =0}^{2k-2}  (-1)^{2k-\ell } \int_{\Omega_T} \partial_{t}^\ell L_\ell u\,\partial_{t}^\ell v\,\dx\dt
- \int_{\Omega_T} \big(\partial^{2k-1}_t L_{2k-1} u +\partial^{2k+1}_t L_{2k} u\big)\, \partial^{2k-1}_t v \,dx\dt. 
\end{aligned}
\end{equation*}
Next, we integrate by parts the last integral on the right-hand side of the previous identity, use  \eqref{BCum} with
\(i=2k-1\), and associate the terms with respect to \(\partial^{2k-2}_t v.\) Repeating this process iteratively, and recalling that \(v(T,\cdot)
= 0\) on \(\Tt^d\), we conclude that
\begin{equation*}
\begin{aligned}
\int_{\Omega_T} \sum_{ |\beta|= 2k}\partial_{t,x}^\beta u\partial_{t,x}^\beta
v\,\dx\dt =\int_{\Omega_T} \sum_{ |\beta|= 2k}\partial_{t,x}^{2\beta} u \,v\,\dx\dt.
\end{aligned}
\end{equation*}
Similarly, if \(w\in H^{2k}(\Omega_T)\) satisfies \(w(0,\cdot)
= 0\) on \(\Tt^d\), we conclude that
\begin{equation*}
\begin{aligned}
\int_{\Omega_T} \sum_{ |\beta|= 2k}\partial_{t,x}^\beta m\partial_{t,x}^\beta
w\,\dx\dt =\int_{\Omega_T} \sum_{ |\beta|= 2k}\partial_{t,x}^{2\beta} m \,w\,\dx\dt.
\end{aligned}
\end{equation*}
This observation is at the core of  Definition~\ref{DMPWS1}.
\end{remark}

%
\begin{remark}\label{rmkmEL}
Let $(m,u)$ be a weak solution to Problem~\ref{MP}. Let $\Omega_T' : = \{(t,x) \in \Omega_T\ |\  m(t,x)>0 \}$, and fix $w_1 \in C_c^\infty(\Omega_T')$. For all $\tau\in \Rr$ with $|\tau|$ small enough, we get $w=m+\tau w_1\in \Aa$.
Then, from (E2), we have
\begin{align*}
& \tau\int_0^T\!\!\!\int_{\Tt^d} \bigg(u_t +\sum_{i,j=1}^da_{ij}(x)u_{x_ix_j} - H(x,Du) +g(m,h(\boldsymbol{m})) + V(t,x) \bigg)w_1 \, \dx\dt
\\&\quad+ \tau\int_0^T\!\!\!\int_{\Tt^d} \bigg[ \epsilon\bigg (m+\sigma+\sum_{|\beta|=2k}
\partial^{2\beta}_{t,x} \sigma\bigg)w_1+\epsilon \sum_{|\beta|=2k}\partial_{t,x}^\beta
m\partial_{t,x}^\beta w_1\bigg]\, \dx\dt \geq 0
.
\end{align*}
Because the sign of $\tau$ is arbitrary, we verify that $m$ satisfies
\begin{align*}
&u_t + \sum_{i,j=1}^da_{ij}(x)u_{x_ix_j}  - H(x,Du) + g(m,h(\boldsymbol{m})) + V(t,x)+\epsilon\bigg (m+\sigma+\sum_{|\beta|=2k}
\partial^{2\beta}_{t,x} (m+\sigma)\bigg)
=0
\end{align*}
pointwise in \(\Omega_T'\). Furthermore, let $w_2\in C_c^\infty(\Omega_T)$ be such that $w_2\geq 0$; then, choosing $w=m+w_2\in \Aa$ in (E2) and integrating by parts, we obtain
\begin{equation*} 
\begin{split}
&\int_0^T\!\!\!\int_{\Tt^d} \bigg(  u_t +\sum_{i,j=1}^da_{ij}(x)u_{x_ix_j} - H(x,Du) +g(m,h(\boldsymbol{m})) 
+ V(t,x)  \bigg)w_2 \, \dx\dt\\
&\quad+\int_0^T\!\!\!\int_{\Tt^d} \epsilon\bigg (m+\sigma+\sum_{|\beta|=2k}
\partial^{2\beta}_{t,x} (m+\sigma)\bigg)
w_2\, \dx\dt\geq 0.
\end{split}
\end{equation*}
Thus,  in the sense of distributions in $\Omega_T$,
\begin{equation*}\label{ELineq1}
u_t  + \sum_{i,j=1}^da_{ij}(x)u_{x_ix_j} - H(x,Du) + g(m,h(\boldsymbol{m})) + 
V(t,x)+\epsilon\bigg (m+\sigma+\sum_{|\beta|=2k}
\partial^{2\beta}_{t,x} (m+\sigma)\bigg)
\geq 0.
\end{equation*}
Also, from (E3), in the sense of distributions in $\Omega_T$, we have
\begin{equation*}\label{Eqdist1}
m_t - \sum_{i,j=1}^d\big(a_{ij}(x)(m+\sigma)\big)_{x_ix_j} - \div\big((m+\sigma) D_p H(x,Du) \big)  
+\epsilon \bigg( u+\xi + \sum_{|\beta|=2k}
\partial^{2\beta}_{t,x}( u+\xi)\bigg )  
= 0.
\end{equation*}
\end{remark}

\begin{teo}\label{MT}
Consider Problem \ref{MP} and suppose that Assumptions \ref{Hconv}--\ref{AWC1} and \ref{Hmono} hold for some $\gamma>1$. Then, there exists a unique weak solution $(m,u) \in H^{2k}(\Omega_T)\times H^{2k}(\Omega_T)$ to Problem \ref{MP} in the sense of Definition~\ref{DMPWS1}.
\end{teo}
The above notions of weak solutions are more relaxed 
than typical weak solutions or classical solutions. However, sometimes, it is possible to show that these weak solutions have further regularity properties.
This matter is examined in Section~\ref{prowsOP}, where we characterized further properties of weak solutions to Problems~\ref{OP} and \ref{MP}. 
Furthermore, in Section~\ref{Secfr}, we show how to extend the monotonicity method to  congestion and density constrained MFGs.

\section{Assumptions}\label{Ass}
Our main results need the following hypotheses on the data in Problems~\ref{OP} and \ref{MP}. These hypotheses are similar to the ones in \cite{FGT1}.
The first four assumptions provide standard convexity and growth conditions on $H$. For example, Assumptions~\ref{Hconv}--\ref{Bderi} hold for
\begin{equation*}
H(x,p)=c(x)|p|^{\gamma}+b(x)\cdot p,
\end{equation*}
where $c\in C^\infty(\Tt^d)$ is positive, $b\in C^\infty(\Tt^d)$, and $\gamma>1$ as in Problem~\ref{OP}.
%
\begin{hyp}\label{Hconv}
For all $x\in \Tt^d$, the map $p\mapsto H(x,p)$ is convex in $\Rr^d$.
\end{hyp}
\begin{remark}\label{Hconv2}
Because of Assumption~\ref{Hconv}, for all $x\in \Tt^d$ and $(p, q) \in \Rr^d\times \Rr^d$, $H$ satisfies
\begin{equation*}
(D_p H(x, p)-D_p H(x,q))\cdot (p-q)\geq 0.
\end{equation*}
\end{remark}

%
\begin{hyp}\label{Hcoer}
There exists a constant, $C>0$, and $\gamma>1$ such that, for all $(x,p)\in \Tt^d\times\Rr^d$,
we have%
\begin{equation*}
\begin{aligned}
- H(x,p)+ D_p H(x,p)\cdot p\geq \frac{1}{C}|p|^\gamma -C.
\end{aligned}
\end{equation*}

\end{hyp}

\begin{hyp}\label{Hbdd}
Let $\gamma>1$ be as in Assumption~\ref{Hcoer}.
There exists a constant, $C>0$, such that, for all $(x,p)\in \Tt^d\times\Rr^d$,
we have%
\begin{equation*}
\begin{aligned}
H(x,p)
\geq 
\frac{1}{C} |p|^\gamma -C.
\end{aligned}
\end{equation*}
\end{hyp}
\begin{hyp}\label{Bderi}
Let $\gamma>1$ be as in Assumption~\ref{Hcoer}.
There exists a constant, $C>0$, such that, for all $(x,p)\in \Tt^d\times\Rr^d$,
we have\begin{equation*}
|D_pH(x,p)|\leq C|p|^{\gamma-1}+C .
\end{equation*}
\end{hyp}

The following assumption is a regularity condition for $h$, see \cite{FG2}. For instance, Assumption~\ref{hhyp} holds for
\begin{equation}\label{exah}
h(\boldsymbol{m})(t,x) 
= 
c(\zeta\ast (\zeta \ast \boldsymbol{m})^{\tau}(\cdot, x))(t)
=c\int_{\Tt^d} \zeta(x-z)\bigg(\int_{\Tt^d} \zeta(z-y) m(y,t)\,
\dy\bigg)^\tau\d z,
\end{equation}
where $c \geq0$ and $\tau>0$
and where
 $\zeta\in C_c^\infty(\Tt^d)$ is such that $\zeta\geq 0$,  
$\int_{\Tt^d} \zeta\,\dx=1$ and $\zeta$ us symmetric.
\begin{hyp}\label{hhyp}
For  each \(\kappa\in\Nn\) such that $\kappa>\frac{d+1}{2}+1$, we have
\begin{align*}
&({\rm a})\ \big\{h(\boldsymbol{m})| m\in H^{\kappa}(\Omega_T; \Rr_0^+)\big\}\subset H^{\kappa}(\Omega_T; \Rr_0^+),\\
&({\rm b})\ m\in H^{\kappa}(\Omega_T; \Rr_0^+)\mapsto h(\boldsymbol{m})\in H^{\kappa}(\Omega_T; \Rr_0^+)\ \mbox{defines a Fr\'{e}chet differentiable map}.
\end{align*}
\end{hyp}

As observed in \cite{FG2}, if \(h\) satisfies  Assumption~\ref{hhyp},
then  for all $\bar m \in H^{\kappa}(\Omega_T; \Rr_0^+)$,
there exists a bounded linear operator,
\(\mathfrak{H}_{\bar m} \in \mathcal{L}
(H^{\kappa}(\Omega_T) ; H^{\kappa}(\Omega_T))
\), such that, for all $m \in  H^{\kappa}(\Omega_T; \Rr_0^+)$,
{\setlength\arraycolsep{0.5pt}
\begin{eqnarray}\label{hFD}
\begin{aligned}
&\Vert h(\boldsymbol{m}) - h(\boldsymbol{\bar
m})\Vert_{H^{\kappa}(\Omega_T)}
\leq \Vert \mathfrak{H}_{\bar
m}\Vert_{\mathcal{L}(H^{\kappa}(\Omega_T) ; H^{\kappa}(\Omega_T))} \Vert m - \bar
m
\Vert_{H^{\kappa}(\Omega_T)} + o\big(\Vert
m - \bar m
\Vert_{H^{\kappa}(\Omega_T)} \big).
\end{aligned}
\end{eqnarray}}%
Therefore, taking $\bar m=0$ in \eqref{hFD},
we get
{\setlength\arraycolsep{0.5pt}
\begin{eqnarray}\label{hbounds}
\begin{aligned}
\Vert h(\boldsymbol{m})\Vert_{H^{\kappa}(\Omega_T)}
\leq C\big(1+ \Vert m\Vert_{H^{\kappa}(\Omega_T)}\big)
\end{aligned}
\end{eqnarray}}%
for some positive constant $C= C
\big (\kappa,\Omega_T, \Vert \mathfrak{H}_{0}\Vert_{\mathcal{L}
(H^{\kappa}(\Omega_T) ; H^{\kappa}(\Omega_T))} ,\Vert h(\boldsymbol{0})
\Vert_{ H^{\kappa}(\Omega_T)} \big)$.

The next three assumptions concern the growth of $g$. For instance, as we discuss in Remark~\ref{esth1} below, Assumptions~\ref{gint}--\ref{gwc} hold for $g(m,\theta)=m^{\tau}+\theta$, $0<\tau\leq 1$, and \(h\) as in \eqref{exah}, which is a standard example in MFGs.
\begin{hyp}\label{gint}
The map $m\mapsto g(m,h(\boldsymbol{m}))$ is monotone with respect to the $ L^2$-inner product; that is, for $m_1, m_2\in L^2(\Omega_T)$, we have
\begin{equation*}
\int_{0^T}\int_{\Tt^d}
(g(m_1,h(\boldsymbol{m_1}))-g(m_2,h(\boldsymbol{m_2})))(m_1-m_2)
\,\dx\dt \geq 0.
\end{equation*}
Moreover, for all $\delta>0$, there exists a positive constant, $C_\delta$, such that, for all $m \in L^1(\Omega_T)$  with \(m\geq 0\), 
we have\begin{equation*}
\max\bigg\{\int_0^T\!\!\!\int_{\Tt^d} |g(m, h(\boldsymbol{m}))|\, \dx\dt,
\int_0^T\!\!\!\int_{\Tt^d} m \,\dx\dt \bigg\}
\leq
\delta \int_0^T\!\!\!\int_{\Tt^d} mg(m,h(\boldsymbol{m}))\, \dx\dt + C_\delta.
\end{equation*}
\end{hyp}

%
%

%
%
\begin{hyp}\label{AWC1}
There exists a constant, $C>0$, such that, for all $m\in L^1(\Omega_T)$
with \(m\geq 0\),
we have
\begin{equation*}
\int_0^T\!\!\!\int_{\Tt^d} mg(m,h(\boldsymbol{m}))\, \dx\dt\geq -C.
\end{equation*}
\end{hyp}
%
%

\begin{hyp}\label{gwc}
If $\{m_j\}_{j=1}^\infty\subseteq L^1(\Omega_T)$ is a sequence
of nonnegative functions satisfying 
\begin{equation*}
\sup_{j\in\Nn} \int_0^T\!\!\!\int_{\Tt^d} m_jg(m_j,h(\boldsymbol{m}_j))\, \dx\dt <+\infty,
\end{equation*}
then there exists a subsequence of $\{m_{j}\}_{j=1}^\infty$ that converges weakly in $L^1(\Omega_T)$.
\end{hyp}

\begin{remark}[On Assumptions~\ref{gint}--\ref{gwc}]\label{esth1}
As we mentioned above, Assumptions~\ref{gint}--\ref{gwc} hold
for $g(m,\theta)=m^{\tau}+\theta$, $\tau>0$, and \(h\) as in \eqref{exah}. To see this, we first note that \(mg(m, h(\boldsymbol{m})) = m^{\tau + 1} +mh(\boldsymbol{m})\). Then,  because \(h(\boldsymbol{m})\geq 0\) for all \(\boldsymbol{m} \in\Mm_{ac}(\Omega_T) \), the only nontrivial condition is the one in Assumption~\ref{gint}. 

To verify that Assumption~\ref{gint} holds, we fix $\delta>0$ and assume that $c=1$, without loss of generality.
We start by observing that there exists a positive constant,
\(C_\delta\),
only depending on \(\delta\) and \(\tau\), such that \(|s|^{\tau} \leq \delta
|s|^{\tau+1} + C_\delta\)  for all \(s\in\Rr\). By symmetry of $\zeta$, for any $f, g\in L^1(\Tt^d)$, we have $\int_{\Tt^d} f(x) (\zeta \ast g)(x) \,\dx=\int_{\Tt^d} (\zeta \ast f)(x) g(x) \,\dx$.
Hence, using the identity $\|\zeta\|_{L^1(\Tt^d)}=1$,
we get
\begin{equation*}
\begin{aligned}
\int_0^T\!\!\!\int_{\Tt^d} h(\boldsymbol{m})\, \dx\dt
&=
\int_0^T\!\!\!\int_{\Tt^d} \big(\zeta\ast (\zeta \ast \boldsymbol{m})^\tau(\cdot,
x)\big)(t)\, \dx\dt=
\int_0^T\!\!\!\int_{\Tt^d}(\zeta*1)(x)\big( (\zeta \ast \boldsymbol{m})^\tau(\cdot,
x)\big)(t)\, \dx\dt\\
&=
\int_0^T\!\!\!\int_{\Tt^d}\big( (\zeta \ast \boldsymbol{m})^\tau(\cdot,
x)\big)(t)\, \dx\dt\leq\delta\int_0^T\!\!\!\int_{\Tt^d}\big( (\zeta \ast \boldsymbol{m})^{\tau+1}(\cdot,
x)\big)(t)\, \dx\dt + C_\delta\\
&  
=
\delta\int_0^T\!\!\!\int_{\Tt^d}\big((\zeta \ast \boldsymbol{m})(\cdot, x) (\zeta
\ast \boldsymbol{m})^{\tau}(\cdot, x)\big)(t)\, \dx\dt + C_\delta\\
&=
\delta\int_0^T\!\!\!\int_{\Tt^d} m(x) \big(\zeta\ast(\zeta \ast \boldsymbol{m})^{\tau}(\cdot,
x)\big)(t)\, \dx\dt + C_\delta
=
\delta\int_0^T\!\!\!\int_{\Tt^d} m h(\boldsymbol{m})\, \dx\dt + C_\delta,
\end{aligned}
\end{equation*}
from which we deduce that Assumption~\ref{gint} holds. \end{remark}

\begin{remark}
In Remark~\ref{esth1}, we consider an explicit example where the nonnegativity
and symmetry conditions on $\zeta$ are crucial. Under Assumption $ \boldsymbol{{\rm
(g1)}}$ in \cite{FG2}, more general cases can be handled.
\end{remark}

Finally, the next assumption imposes the monotonicity of the functional in Definition~\ref{DOPWS1}. 
Monotonicity is crucial in the proof of Theorem~\ref{TOP} and Theorem~\ref{MT} through Minty's method.
\begin{hyp}\label{Hmono}
The functional $F$ introduced in Definition~\ref{DOPWS1} is monotone with respect to the $L^2\times L^2$-inner product; that is,
for all $(\eta_1, v_1)$, $(\eta_2, v_2) \in \Aa   \times \Bb$, $F$ satisfies
\begin{equation*}
\left\langle
   F
  \begin{bmatrix}
      \eta_1  \\
      v_1 
  \end{bmatrix}
  -
   F
  \begin{bmatrix}
      \eta_2  \\
      v_2 
  \end{bmatrix},
  \begin{bmatrix}
      \eta_1  \\
      v_1
  \end{bmatrix}
  -
  \begin{bmatrix}
      \eta_2  \\
      v_2
  \end{bmatrix}
\right\rangle
\geq 0.
\end{equation*}
\end{hyp}

%
%
\section{Properties of weak solutions}\label{Prows}
Here, we examine the properties of weak solutions, $(m,u)$, to Problem~\ref{MP}. As in \cite{FGT1}, we prove a priori estimates for classical solutions and weak solutions. Moreover, we establish that $u$ belongs to $L^{\gamma}((0,T); W^{1,\gamma}(\Tt^d))$ and that $(\sqrt{\epsilon}m,\sqrt{\epsilon} u)$ is bounded in $H^{2k}(\Omega_T)\times H^{2k}(\Omega_T)$, independently of $\epsilon$.

To simplify the notation, throughout this section, we write the same letter $C$ to denote any positive constant  depending only  on the data; that is, depending only on \(\Omega_T\), \(d\), \(\gamma\), $H$, $V$, $\sigma$, $\xi$,  $m_0$,  $u_T$, on the constants in the Assumptions~\ref{Hcoer}--\ref{AWC1}, and
on constants such as the constants in Morrey's theorem or in the Gagliardo--Nirenberg interpolation inequality.
In particular, these constants are independent of the choice of solutions to Problem~\ref{MP} and of $\epsilon$.
\begin{pro}\label{apriori1}
Consider Problem~\ref{MP} and suppose that Assumptions~\ref{Hcoer}--\ref{gint} hold for some \(\gamma>1\). Then, there exists a positive constant, $C$, depending only on the problem data, such that any classical solution $(m, u)$ to Problem~\ref{MP} satisfies
\begin{equation}\label{aprimu}
\begin{aligned}
&\int_0^T\!\!\!\int_{\Tt^d} \Big(
mg(m,h(\boldsymbol{m}))+\frac{1}{C} (m+\sigma)|Du|^\gamma +\frac{1}{C} m_{0} |Du|^\gamma \Big)\, \dx\dt\\
&\quad+
{\epsilon} \int_0^T\!\!\!\int_{\Tt^d} \
\bigg(
m^2 + u^2 + \sum_{|\beta|=2k}(\partial_{t,x}^\beta m)^2
+ \sum_{|\beta|= 2k}(\partial_{t,x}^\beta u)^2
 \bigg) \, \dx\dt
\leq C\big (1+\|Du\|_{L^1(\Omega_T)}\big). 
\end{aligned}
\end{equation}
\end{pro}
%
%
\begin{proof}
Multiplying the first equation in \eqref{HighdRP} by $(m-m_{0})$ and the second one by $(u-u_T)$, 
adding and integrating over $\Omega_T$, and then integrating by parts and taking the boundary conditions into account, we obtain
\begin{equation}\label{E3.2}
 \begin{aligned}
  &\int_0^T\!\!\!\int_{\Tt^d} \bigg[
  m g(m,h(\boldsymbol{m}))+(m+\sigma)  \big( - H(x,Du) + D_p H(x,Du) \cdot Du \big) \\
  &\quad+ m_{0} H(x,Du) 
  + 
  \epsilon \bigg( m^2 + u^2 + \sum_{|\beta|=2k}(\partial_{t,x}^\beta m)^2 
  + \sum_{|\beta|= 2k}(\partial_{t,x}^\beta u)^2\bigg)
  \bigg] \, \dx\dt\\
   = &\int_0^T\!\!\!\int_{\Tt^d}\bigg[
  -\sum_{i,j=1}^du_{x_i} \big(a_{ij}(m_0+\sigma)\big)_{x_j}+ m_{0} g(m,h(\boldsymbol{m}))
  +\bigg(-\sum_{i,j=1}^d\big(a_{ij}{u_T}_{x_ix_j}\big) - V \bigg) m\\
 &\quad + \bigg( V m_{0} -\sum_{i,j=1}^d(a_{ij}\sigma)_{x_ix_j}u_T+\epsilon \sigma m_0  + \epsilon \xi u_T \bigg) \\
  &\quad + ( m+\sigma ) D_p H(x,Du) \cdot Du_T 
   - \sigma H(x,Du) + \epsilon \big( u(u_T-\xi) +m(m_0-\sigma)\big) \\
  &\quad + \epsilon\bigg( \sum_{|\beta|=2k} 
   \partial_{t,x}^\beta m \partial_{t,x}^\beta m_0 
   -\sum_{|\beta|= 2k} \partial_{t,x}^{2\beta}\sigma (m-m_0) +
  \sum_{|\beta|= 2k} \partial_{t,x}^\beta u \partial_{t,x}^\beta u_T -\sum_{|\beta|= 2k} \partial_{t,x}^{2\beta}\xi (u- u_T )\bigg)\bigg]\, \dx\dt.
  \end{aligned}
\end{equation}
From Assumptions~\ref{Hcoer}--\ref{Bderi}, Young's inequality, and the positivity of $m$, $\sigma$, and $m_0$, we get
\begin{equation*}
\begin{aligned}
&\int_0^T\!\!\!\int_{\Tt^d} ( m+\sigma ) \big( - H(x,Du) + D_p H(x,Du) \cdot Du\big)\, \dx\dt \geq 
\int_0^T\!\!\!\int_{\Tt^d}\Big( \frac{(m+\sigma)|Du|^\gamma}{C} -C(m+\sigma)\Big)\, \dx\dt,
\end{aligned}
\end{equation*}
\begin{equation*}
\begin{aligned}
&\int_0^T\!\!\!\int_{\Tt^d} m_{0} H(x,Du) \,\dx\dt 
\geq 
\int_0^T\!\!\!\int_{\Tt^d} \bigg( \frac{m_{0} |Du|^\gamma}{C} - C m_{0} \bigg)\,\dx\dt,
\end{aligned}
\end{equation*}
\begin{equation*}
\begin{aligned}
\int_0^T\!\!\!\int_{\Tt^d} ( m+\sigma ) D_p H(x,Du) \cdot Du_T \,\dx\dt
&\leq
\int_0^T\!\!\!\int_{\Tt^d} C ( m+\sigma )( |Du|^{\gamma-1} + 1 )\,\dx\dt\\
&
\leq\int_0^T\!\!\!\int_{\Tt^d}\bigg( \frac{( m+\sigma ) |Du|^\gamma}{2C} + C m \bigg) \,\dx\dt + C,
\end{aligned}
\end{equation*}
and
\begin{equation*}
\begin{aligned}
&-\int_0^T\!\!\!\int_{\Tt^d} \sigma H(x,Du)\, \dx\dt 
\leq \int_0^T\!\!\!\int_{\Tt^d}(- \frac{\sigma|Du|^\gamma}{C} + C \sigma )\, \dx\dt \leq C.
\end{aligned}
\end{equation*}
Using these estimates in \eqref{E3.2}, together with Young's inequality and Assumption~\ref{gint},  we obtain
\begin{equation*}
 \begin{split}
   &\int_0^T\!\!\!\int_{\Tt^d} \bigg[ mg(m,h(\boldsymbol{m})) + \frac{(m+\sigma)|Du|^\gamma}{C} +  \frac{m_{0} |Du|^\gamma}{C} \\
   &\quad
   + \frac{\epsilon}{2} \bigg( m^2 + u^2 + \sum_{|\beta|=2k}
   (\partial_{t,x}^\beta m)^2 
   + \sum_{|\beta|= 2k} (\partial_{t,x}^\beta u)^2 \bigg)\bigg] \, \dx\dt\\
   \leq
   &\int_0^T\!\!\!\int_{\Tt^d}\Big( 
   m_{0} g(m,h(\boldsymbol{m}))+ C m +  \frac{(m+\sigma)|Du|^\gamma}{2C} \Big)\, \dx\dt 
   + C\big(1+ \|Du\|_{L^1(\Omega_T)}\big)\\
   \leq 
  &\int_0^T\!\!\!\int_{\Tt^d}\Big( \frac{mg(m,h(\boldsymbol{m}))}{2} + \frac{(m+\sigma)|Du|^\gamma}{2C} \Big) \, \dx\dt + C\big(1+ \|Du\|_{L^1(\Omega_T)}\big),
 \end{split}
\end{equation*}
from which the conclusion follows.
\end{proof}

The preceding result can be extended to weak solutions of Problem~\ref{MP} in the sense of Definition~\ref{DMPWS1}.
\begin{pro}\label{apriWS1}
Consider Problem \ref{MP} and suppose that Assumptions~\ref{Hcoer}--\ref{gint} hold for some \(\gamma>1\). Then, any weak solution $(m, u)$ to Problem~\ref{MP} in the sense of Definition~\ref{DMPWS1} satisfies \eqref{aprimu}. 
\end{pro}
\begin{proof}
Let $(m, u)$ be a weak solution to Problem~\ref{MP} in the sense of Definition~\ref{DMPWS1}. 
Using the properties (E2) and  (E3) in Definition~\ref{DMPWS1} with  $v = u_T -u\in \Bb_0$ and $w= m_0 \in \Aa$ and adding the resulting inequalities, we obtain \eqref{E3.2} with ``$=$'' replaced by ``$\leq$". Consequently, arguing as in the proof of Proposition~\ref{apriori1}, we obtain that $(m, u)$ satisfies \eqref{aprimu}.
\end{proof}


%
%
\begin{cor}\label{apriDu1}
Consider Problem \ref{MP} and suppose that Assumptions~\ref{Hcoer}--\ref{AWC1} hold for some \(\gamma>1\). Then, there exists a positive constant, $C$, depending only on the
problem data, such that any weak solution $(m,u)$ to Problem~\ref{MP} satisfies \(\Vert Du\Vert_{L^\gamma(\Omega_T)}\leq C\).
\end{cor}
\begin{proof}
Because $m_0$ is strictly positive, we have $c:=\min_{\Tt^d}m_0>0$ and, using Proposition~\ref{apriWS1} with Assumption~\ref{AWC1} and Young's inequality with $\gamma>1$, 
we obtain
\begin{align*}
-C + \frac{c}{C}\int_0^T\!\!\!\int_{\Tt^d}|Du(t,x)|^\gamma\,\dx\dt
&\leq
\int_0^T\!\!\!\int_{\Tt^d}\Big(mg(m,h(\boldsymbol{m}))+\frac{m_0}{C}|Du(t,x)|^\gamma\Big)\, \dx\dt
\\
&\leq C\big(1+\|Du\|_{L^1(\Omega_T)}\big)
\leq C +\frac{c}{2C}\int_0^T\!\!\!\int_{\Tt^d}|Du(t,x)|^\gamma\, \dx\dt.
\qedhere
\end{align*}
\end{proof}

\begin{cor}\label{aprisqrt}
Consider Problem \ref{MP} and suppose that
Assumptions~\ref{Hcoer}--\ref{AWC1}
 hold for some \(\gamma>1\). Then, there exists a positive constant,
$C$, depending only on the
problem data, such that any weak solution $(m,u)$ to Problem~\ref{MP}  satisfies \(\Vert \sqrt{\epsi} m\Vert_{H^{2k}(\Omega_T)} + \Vert \sqrt{\epsi} u\Vert_{H^{2k}(\Omega_T)}\leq C\).
\end{cor}
\begin{proof}
Using  Proposition~\ref{apriWS1},
Assumption~\ref{AWC1}, and the positivity of $m$, $\sigma$, and $m_0$,  we obtain
\begin{equation*}
\begin{aligned}
{\epsilon} \int_0^T\!\!\!\int_{\Tt^d} \
\bigg(
m^2 + u^2 + \sum_{|\beta|=2k}(\partial_{t,x}^\beta m)^2
+ \sum_{ |\beta|= 2k}(\partial_{t,x}^\beta u)^2
 \bigg) \, \dx\dt
\leq C\big(1+ \|Du\|_{L^1(\Omega_T)}\big),
\end{aligned}
\end{equation*}
where \(C\) is a positive constant depending only on the
problem data. 
From Corollary~\ref{apriDu1}, 
Corollary~\ref{aprisqrt} follows.
\end{proof}

\section{A variational problem}\label{VP}
In this section, we investigate a variational problem whose Euler--Lagrange equation is related to the first equation in \eqref{HighdRP}. We show that there exists a unique minimizer, $m$, to this problem. Also, we examine properties of $m$ from which deduce the existence and uniqueness of a weak solution to Problem~\ref{MP}. 

Set $\widehat \sigma:=\sigma +m_0$ and $\widehat g(m,\widehat h(\boldsymbol{m})):=g(m+m_0,h(\boldsymbol{m}+\boldsymbol{m}_0))$. Given $(m, u) \in H^{2k-2}(\Omega_T) \times H^{2k-1}(\Omega_T)$ with $ m+m_0 \geq 0$, let $I_{(m, u)}: H^{2k}(\Omega_T) \to \Rr$, for $w\in H^{2k}(\Omega_T)$ be given by
\begin{equation}\label{defImu}
\begin{split}
I_{(m, u)}[w]: =&\,  \int_0^T\!\!\!\int_{\Tt^d} \bigg[ \frac{\epsilon }{2} \Big( \Big(w+\widehat\sigma + \sum_{|\beta|=2k}
\partial^{2\beta}_{t,x} \widehat\sigma\Big)^2 + \sum_{|\beta|=2k}(\partial_{t,x}^\beta w)^2 \Big)\\
&\quad+ 
\Big( u_t + \sum_{i,j=1}^da_{ij}u_{x_i x_j} - H(x,Du) + \widehat g(m,\widehat h(\boldsymbol{m}))- V \Big) w \bigg] \, \dx\dt.
\end{split}
\end{equation}
Next, we fix $(m_1,u_1)\in H^{2k-2}(\Omega_T)\times H^{2k-1}(\Omega_T)$ with $m_1+m_0\geq0$, and set $I_1=I_{(m_1,u_1)}$. We address the variational problem  of finding $m\in \Aa_0$ such that 
%
\begin{equation}\label{VP1}
I_1[m] =\inf_{w \in \Aa_0}I_1[w],
\end{equation}
where $\Aa_0$ is defined in \eqref{DAa0}.
%
%
\begin{pro}\label{EMVP1} 
Let $H$, $g$, \(h\), $\sigma$,  $V$, \(\{a_{ij}\}_{i,j=1}^d\),
and \(m_0\) be as in Problem~\ref{MP}, and 
fix $(m_1, u_1) \in H^{2k-2}(\Omega_T) \times H^{2k-1}(\Omega_T)$ such that $m_1+m_0 \geq 0$. Then, there exists a unique $m \in \Aa_0$ satisfying \eqref{VP1}.
\end{pro}
%
%
\begin{proof}  
Invoking Young's inequality, $I_1[\cdot]$ is bounded from below and the bound depends on the problem data, $\epsilon$, $m_1$, and $u_1$. Thus, also taking
\(w=0\) as test function in \eqref{VP1},
we conclude that  the infimum in \eqref{VP1} is finite.

Let $\{w_{n}\}_{n=1}^{\infty} \subset \Aa_0$ be a minimizing sequence for \eqref{VP1}, and fix $\delta\in(0,1)$. 
Then, there exists $N\in \Nn$ 
such that, for all $n \geq N$, 
\begin{equation}\label{Bvp1}
I_1[w_n]<\inf_{w\in\Aa_0} I_1[w] +\delta \leq I_1[0] + 1=C. 
\end{equation}
By Morrey's embedding theorem, $H^{2k-3}(\Omega_T)$ is compactly embedded
in \(C^{0,l}(\overline \Omega_T)\) for
some \(l\in (0,1)\). In particular,
there exists a positive constant, \(C=C(\Omega_T,k,d,l)\),
such that,
for all \(\vartheta\in H^{2k-3}(\Omega_T)\),
we have
\begin{equation}
\label{eq:MorreyET}
\begin{aligned}
\Vert \vartheta \Vert_{C^{0,l}(\overline \Omega_T)} \leq C \Vert \vartheta \Vert_{H^{2k-3}(
\Omega_T)}.
\end{aligned}
\end{equation}
From \eqref{eq:MorreyET}, we get 
$m_1\in C^{1,l}(\overline\Omega_T)$ and $u_1 \in C^{2,l }(\overline\Omega_T)$ for some $l\in(0,1)$, and  
$C_0:=\max\{\|\widehat\sigma\|_{C^{4k}(\Omega_T)},\| {u_1}_t + \sum_{i,j=1}^da_{ij}{u_1}_{x_ix_j} - H(\cdot,Du_1) +\widehat g(m_1,\widehat h(\boldsymbol{m}_1))- V \|_{L^\infty(\Omega_T)}\} <\infty$.
Then, by Young's inequality and \eqref{Bvp1}, for all $n\geq N$, we obtain
\begin{equation}\label{eqvp-3}
\begin{split}
\frac{\epsilon }{2}\int_0^T\!\!\!\int_{\Tt^d} \Big( w_{n}^2 + \sum_{|\beta|=2k}
(\partial_{t,x}^\beta  w_n)^2\Big) \, \dx\dt 
&\leq \int_0^T\!\!\!\int_{\Tt^d}C_0 (\epsilon +1 ) |w_n| \, \dx\dt + C\\
&\leq 
\frac{\epsilon }{4} \int_0^T\!\!\!\int_{\Tt^d}\Big( w_n^2 
+ 
\sum_{|\beta|=2k}(\partial_{t,x}^{\beta} w_n)^2\Big)\, \dx\dt 
+\frac{C}{\epsilon}+C.
\end{split}
\end{equation}
Invoking the Gagliardo--Nirenberg interpolation inequality, we get
\begin{equation}\label{eqvp-4}
\|\partial_{t,x}^\beta w_n\|_{L^2(\Omega_T)}^2 
\leq 
C \big(\|w_n\|_{L^2(\Omega_T)}^2 +\|D_{t,x}^{2k}w_n\|_{L^2(\Omega_T)}^2\big),
\end{equation}
where $\beta\in\Nn_0^{d+1}$ is any multi-index such that $|\beta|\leq 2k$.
Hence, by \eqref{eqvp-3} and \eqref{eqvp-4}, we obtain that $\{w_n\}_{n=1}^{\infty}$ is bounded in $H^{2k}(\Omega_T)$.
Therefore, $w_n \rightharpoonup m$ weakly in $H^{2k}(\Omega_T)$ for some $m\in H^{2k}(\Omega_T)$, extracting a subsequence if necessary. 
Furthermore, by Morrey's embedding theorem, $w_n \to m$ in $C^{2,l}(\overline \Omega_T)$ for some $l\in (0,1)$.
Consequently, because $w_n +m_0 \geq 0$ and $w_n(0,x)=0$, also $m + m_0 \geq0$ and $m(0,x)=0$. Thus,   $m \in \Aa_0$. 
Also, $w_n\to m$ in $L^2(\Omega_T)$ and $\Vert D_{t,x}^{2k} m\Vert_{L^2(\Omega_T)}^2 \leq \liminf_{n \to \infty} \Vert D_{t,x}^{2k} w_n \Vert_{L^2(\Omega_T)}^2$; hence, $I_1[m]\leq \liminf_n I_1[w_n]=\inf_{w \in \Aa_0}I_1[w]\leq I_1[m]$, from which we conclude that $m$ is a minimizer of $I_1$ over $\Aa_0$.

Next, we verify uniqueness. Suppose that $m$, $\widetilde{m} \in \Aa_0$ are minimizers of $I_1$ over $\Aa_0$
 with $m \neq \widetilde{m}$. Then, $\frac{m + \widetilde{m}}{2} \in \Aa_0$,
 $m-\widetilde{m}\in C^0(\overline\Omega_T)$,
and $\int_0^T\!\!\!\int_{\Tt^d}(m-\widetilde{m})^2\,\dx\dt>0$.
  Thus,
\begin{equation}\label{eqvp-5}  
 \begin{aligned}
 I_1\left[\frac{m + \widetilde{m}}{2}\right]
 &=\int_0^T\!\!\!\int_{\Tt^d} \bigg[ \frac{\epsilon }{2}
  \Big( \Big(\frac{m+\widetilde{m}}{2}+\widehat\sigma+\sum_{|\beta|=2k}
\partial^{2\beta}_{t,x}\widehat\sigma\Big)^2 
  + \sum_{|\beta|=2k}\Big(\frac{\partial_{t,x}^\beta m+\partial_{t,x}^\beta \widetilde{m}}{2}\Big)^2\Big)\\
   &\quad\quad+ \Big( {u_1}_t + \sum_{i,j=1}^da_{ij}(x){u_1}_{x_ix_j}  - H(x,Du_1) +\widehat g(m_1,\widehat h(\boldsymbol{m}_1))- V \Big)
    \Big(\frac{m+\widetilde{m}}{2}\Big) \bigg]\, \dx\dt \\ 
 &=\frac{1}{2}\int_0^T\!\!\!\int_{\Tt^d} \bigg[
     \frac{\epsilon }{2} \Big( \Big(m+\widehat\sigma+\sum_{|\beta|=2k}
\partial^{2\beta}_{t,x}\widehat\sigma\Big)^2 
     + \sum_{|\beta|=2k}(\partial_{t,x}^\beta m )^2\Big)\\
   &\quad \quad+ \Big(  {u_1}_t + \sum_{i,j=1}^da_{ij}(x){u_1}_{x_ix_j} - H(x,Du_1) +\widehat g(m_1,\widehat h(\boldsymbol{m}_1))- V \Big)m \bigg] \, \dx\dt \\
   &\quad+\frac{1}{2}\int_0^T\!\!\!\int_{\Tt^d} \bigg[
    \frac{\epsilon }{2} \Big( \Big(\widetilde{m}+\widehat\sigma+\sum_{|\beta|=2k}
\partial^{2\beta}_{t,x}\widehat\sigma\Big)^2 
    + \sum_{|\beta|=2k}(\partial_{t,x}^\beta \widetilde{m})^2 \Big)\\
   &\quad \quad+ \Big( {u_1}_t + \sum_{i,j=1}^da_{ij}(x){u_1}_{x_ix_j} - H(x,Du_1) +\widehat g(m_1,\widehat h(\boldsymbol{m}_1))- V \Big)\widetilde{m} \bigg] \, \dx\dt \\
  &\quad- \frac{\epsilon}{8}\int_0^T\!\!\!\int_{\Tt^d} \Big[ (m - \widetilde{m})^2  + \sum_{|\beta|=2k}(\partial_{t,x}^\beta m - \partial_{t,x}^\beta \widetilde{m})^2  \Big] \, \dx\dt
\\  &<\frac{1}{2}I_1[m] + \frac{1}{2}I_1[\widetilde{m}] = \min_{w \in \Aa_0}I_1[w],
\end{aligned}
\end{equation}
which contradicts the fact that $m$ and $\widetilde{m}$ are minimizers.
Hence, we have $m=\widetilde{m}$.
\end{proof}

\begin{cor}\label{aprimvp1}
Let $H$, $g$, \(h\), $\sigma$,  $V$, \(\{a_{ij}\}_{i,j=1}^d\),
and \(m_0\) be as in Problem~\ref{MP}, fix  $(m_1, u_1) \in H^{2k-2}(\Omega_T) \times H^{2k-1}(\Omega_T)$ with $m_1+m_0 \geq 0$, and let $m \in \Aa_0$ be the unique solution to \eqref{VP1}. Set $C_0:=\|{u_1}_t + \sum_{i,j=1}^da_{ij}{u_1}_{x_ix_j} - H(\cdot,Du_1) +\widehat g(m_1,\widehat h(\boldsymbol{m}_1))- V \|_{L^\infty(\Omega_T)}$. Then, there exists a positive constant, $C$, depending only on the problem data
and on \(C_0\), such that
 $\Vert m\Vert_{H^{2k}(\Omega_T)} \leq C $.
\end{cor}
\begin{proof}
As $I_1[m]\leq I_1[0] $, it follows that \eqref{eqvp-3} and \eqref{eqvp-4} hold with $w_n$ replaced by $m$,
 which yields the conclusion.
\end{proof}

\begin{pro}\label{PVI}
Let $H$, $g$,  \(h\), $\sigma$,  $V$, \(\{a_{ij}\}_{i,j=1}^d\),
and \(m_0\) be as in Problem \ref{MP}, fix
$(m_1, u_1) \in H^{2k-2}(\Omega_T) \times H^{2k-1}(\Omega_T)$ with $m_1+m_0\geq 0$, and let $m \in \Aa_0$ be the unique solution to \eqref{VP1}. Then, for any $w \in \Aa_0$, $m$ satisfies
\begin{equation}\label{VI1}
\begin{aligned}
&\int_0^T\!\!\!\int_{\Tt^d} \Big( {u_1}_t + \sum_{i,j=1}^da_{ij}(x){u_1}_{x_ix_j} - H(x,Du_1) +\widehat g(m_1,\widehat h(\boldsymbol{m}_1))- V \Big)( w - m )  \, \dx\dt \\
&\quad+\int_0^T\!\!\!\int_{\Tt^d}\Big[\epsilon\Big ( m+\widehat\sigma +\sum_{|\beta|=2k}
\partial^{2\beta}_{t,x}\widehat\sigma\Big ) ( w - m ) 
+ \epsilon \sum_{|\beta|=2k}\partial_{t,x}^\beta m (\partial_{t,x}^\beta w-\partial_{t,x}^\beta m)\Big] \, \dx\dt 
\geq 0.
\end{aligned}
\end{equation}
\end{pro}
%
%
\begin{proof} 
Let $w \in \Aa_0$. For $\tau \in [0,1]$, we have $m + \tau(w - m)=( 1 - \tau)m + \tau w \in \Aa_0$.
Thus, the mapping $i: [0,1] \to \Rr$ given by
\begin{equation*}
i[\tau] := I_1 \big[m + \tau(w - m)\big]
\end{equation*}
is a $C^\infty$-function. Because 
 $i(0)\leq i(\tau)$ for all $0\leq \tau \leq 1$, we have $i'(0) \geq 0$. 
On the other hand, for $0<\tau\leq1$,
we get
\begin{equation*}
 \begin{aligned}
\frac{1}{\tau}\big(i(\tau)-i(0)\big)=&\,\int_0^T\!\!\!\int_{\Tt^d} \Big ( {u_1}_t + \sum_{i,j=1}^da_{ij}(x){u_1}_{x_ix_j} - H(x,Du_1) +\widehat g(m_1,\widehat h(\boldsymbol{m}_1))-V \Big)( w - m ) \, \dx\dt \\
&\quad + {\epsilon} \int_0^T\!\!\!\int_{\Tt^d} \Big[
\Big( m+\widehat\sigma +\sum_{|\beta|=2k}
\partial^{2\beta}_{t,x}\widehat\sigma \Big)( w - m) +\sum_{|\beta|=2k}\partial_{t,x}^\beta
m (\partial_{t,x}^\beta w - \partial_{t,x}^\beta
m)\Big] \, \dx\dt\\
&\quad +\frac{\epsilon\tau}{2} \int_0^T\!\!\!\int_{\Tt^d} \Big[( w - m)^2+  \sum_{|\beta|=2k}(\partial_{t,x}^\beta w -\partial_{t,x}^\beta m )^2\Big] \, \dx\dt.
 \end{aligned}
\end{equation*}
Thus, letting $\tau\to 0^+$ in the preceding inequality and using the inequality $i'(0)\geq 0$, we obtain \eqref{VI1}.
\end{proof}

\begin{pro}
Let $H$, $g$, \(h\), $\sigma$,  $V$, \(\{a_{ij}\}_{i,j=1}^d\),
and \(m_0\) be as in Problem~\ref{MP}, fix $(m_1, u_1) \in H^{2k-2}(\Omega_T) \times H^{2k-1}(\Omega_T)$ with $m_1+m_0 \geq 0$, and let $m$ be the unique solution of \eqref{VP1}. Set $\widehat\Omega_T = \{ (t,x) \in \Omega_T\ |\  m(t,x) +m_0(x) > 0 \}$.
Then,  $m$ satisfies
\begin{align*}
&
{u_1}_t + \sum_{i,j=1}^da_{ij}(x){u_1}_{x_ix_j} - H(x,Du_1) +\widehat g(m_1,\widehat h(\boldsymbol{m}_1)) - V 
+\epsilon \Big( m+\widehat\sigma+ \sum_{|\beta|=2k}
\partial^{2\beta}_{t,x}( m+\widehat\sigma ) \Big)= 0 
\end{align*}
pointwise in \(\widehat\Omega_T\) and
\begin{equation*}
{u_1}_t + \sum_{i,j=1}^da_{ij}(x){u_1}_{x_ix_j} - H(x,Du_1) +\widehat g(m_1,\widehat h(\boldsymbol{m}_1)) -V+\epsilon\Big ( m+\widehat\sigma + \sum_{|\beta|=2k}
\partial^{2\beta}_{t,x}( m+\widehat\sigma ) \Big)\geq 0 
\end{equation*}
in the sense of distributions in $\Omega_T$.
\end{pro}
\begin{proof}
To verify the statement, it is enough to argue as in Remark~\ref{rmkmEL}, using \eqref{VI1} instead of (E2) and recalling the embedding $H^{2k-2}(\Omega_T)\hookrightarrow C^{1,l}(\overline\Omega_T)$ for some $l\in(0,1)$.
\end{proof}

\section{A problem given by a bilinear form}\label{BF}
Here, we consider a problem given by a bilinear form associated with the second equation in \eqref{HighdRP}. We use Lax--Milgram theorem to show that there exists a unique solution, $u$, to this problem. Also, we establish a uniform bound for $u$. In Section~\ref{PfMT1}, we apply these results to prove that there exists a unique weak solution to Problem~\ref{MP}.

Let $\Bb_0$ be as in \eqref{DBb0}. Suppose that  $H$,  $\sigma$,   \(\{a_{ij}\}_{i,j=1}^d\),
 \(m_0\), and $\xi$ are as in Problem~\ref{MP} and, as in the previous section, let $\widehat \sigma =\sigma+m_0$. 
Given $(m, u) \in H^{2k-2}(\Omega_T) \times H^{2k-1}(\Omega_T)$ with $m+m_0 \geq 0$, we define a bilinear form, $B:\Bb_0 \times \Bb_0 \to \Rr$, and a linear functional, $f_{(m, u)}:H^{2k}(\Omega_T) \to \Rr$, by setting, 
\begin{equation}
\label{defBfmu}
\begin{aligned}
B[v_1,v_2]:=&\, \epsilon \int_0^T\!\!\!\int_{\Tt^d}  \Big( v_1v_2 +\sum_{|\beta|=2k}\partial_{t,x}^\beta v_1 \partial_{t,x}^\beta v_2 \Big)\, \dx\dt,\\
\big\langle f_{(m, u)},v\big\rangle
:=&\,\int_0^T\!\!\!\int_{\Tt^d} \Big[- m_t + \sum_{i,j=1}^d\big(a_{ij}(x)(m+\widehat\sigma)\big)_{x_ix_j} \\
& \quad
+ \div\big((m+\widehat\sigma) D_p H(x,Du)\big) - 
\epsilon \Big(\xi+\sum_{|\beta|=2k}\partial_{t,x}^{2\beta}
\xi\Big) \Big]v \, \dx\dt
\end{aligned}
\end{equation}
for $v_1,v_2\in \Bb_0$ and $v\in H^{2k}(\Omega_T)$.

Fix $(m_1,u_1)\in H^{2k-2}(\Omega_T) \times H^{2k-1}(\Omega_T)$ with \(m_1+m_0\geq0\), and take $f_1:= f_{(m_1,u_1)}$. Next, we study the problem of finding $u\in \Bb_0$ satisfying\begin{equation}\label{BL1}
B[u,v]=\langle f_1,v \rangle \quad \mbox{for all}\ v \in \Bb_0.
\end{equation}

\begin{pro}\label{ESBP1}
Let $H$,  $\sigma$,
  \(\{a_{ij}\}_{i,j=1}^d\),
 \(m_0\), and $\xi$ be as in Problem~\ref{MP}, and fix $(m_1, u_1) \in H^{2k-2}(\Omega_T) \times H^{2k-1}(\Omega_T)$ with $m_1+m_0\geq0$. Then, there exists a unique  solution, $u \in \Bb_0$, to \eqref{BL1}. Moreover, there exists a positive constant, $C$,
depending only on the problem data,
on \(\epsilon\), on $\Vert m_1\Vert_{H^{2k-2}(\Omega_T)}$, and  
on $\Vert u_1\Vert_{H^{2k-1}(\Omega_T)}$,
such that
$\Vert u\Vert_{H^{2k}(\Omega_T)} \leq C$.
\end{pro}
\begin{proof}
Because $(m_1,u_1)\in \big(H^{2k-2}(\Omega_T) \times H^{2k-1}(\Omega_T)\big) \cap \big( C^{1,l}(\overline\Omega_T) \times 
C^{2,l}(\overline\Omega_T)\big)$ for some $l\in(0,1)$ (see \eqref{eq:MorreyET}),
we obtain $\big( - {m_1}_t + \sum_{i,j=1}^d\big(a_{ij}(x)(m+\widehat\sigma)\big)_{x_ix_j} + \div\big( (m_1+\widehat\sigma) D_p H(x,Du_1)\big) -\epsilon \big(\xi+\sum_{|\beta|=2k}\partial_{t,x}^{2\beta}
\xi\big)\big) \in L^2(\Omega_T)$. Hence, from H\"older's inequality, $f_1$ is bounded in \(L^2(\Omega_T)\). 

By Cauchy--Schwarz inequality, we get $|B[v_1,v_2]| \allowbreak \leq \epsilon \Vert v_1\Vert_{H^{2k}(\Omega_T)}\Vert v_2\Vert_{H^{2k}(\Omega_T)}$ for all   $v_1,v_2\in \Bb_0$. 
Furthermore, by the Gagliardo--Nirenberg interpolation
inequality (see \eqref{eqvp-4}), we have $B[v_1,v_1]\geq \epsi C \Vert v_1\Vert_{H^{2k}(\Omega_T)}^2$ for all   $v_1\in \Bb_0$. 
By applying the Lax--Milgram theorem to \eqref{BL1}, there exists a unique solution, $u\in \Bb_0$, to \eqref{BL1}.

Since $c_0:=\big\Vert - {m_1}_t + \sum_{i,j=1}^d(a_{ij}(x)m)_{x_ix_j} + \div\big((m_1+\widehat\sigma)D_p H(x,Du_1)\big) - \epsilon \big(\xi+\sum_{|\beta|=2k}\partial_{t,x}^{2\beta}
\xi\big)  \big\Vert_{ L^2(\Omega_T)}^2 <\infty$, from Young's 
inequality and the Gagliardo--Nirenberg
interpolation
inequality, we have
\begin{equation*}
\epsilon C\Vert u\Vert_{H^{2k}(\Omega_T)}^2
\leq
B[u,u] = \langle f_1, u \rangle
\leq
\frac{\epsilon C}{2}\Vert u\Vert_{L^2(\Omega_T)}^2 + \frac{c_0}{4C\epsilon}.
\end{equation*}
Therefore, we have $\Vert u\Vert_{H^{2k}(\Omega_T)}^2
 \leq \frac{c_0}{4(C\epsi)^2} $, from which Lemma~\ref{ESBP1} follows.
\end{proof}

\section{Proof of Theorem~\ref{MT}}\label{PfMT1}
Here, we prove Theorem~\ref{MT}. First, by Schaefer's fixed-point theorem, we verify that there exists a unique weak solution to \eqref{HighdRP} with $u_T\equiv0$.
Next, we generalize this result for any $u_T\in C^{4k}(\Tt^d)$.

Suppose that $u_T\equiv 0$. As in Sections~\ref{VP} an \ref{BF}, let  $\widehat g(m,\widehat h(\boldsymbol{m}))=g(m+m_0, h(\boldsymbol{m}+\boldsymbol{m_0}))$ and \(\widehat \sigma = \sigma + m_0\).
Let $\widetilde\Aa_0$ and $\widetilde\Bb_0$ be the sets containing \(\Aa_0\) and $\Bb_0$ (see  \eqref{DAa0}
and \eqref{DBb0}), respectively,  given by
\begin{align*}
&\widetilde\Aa_0:=\{w \in H^{2k-2}(\Omega_T)\ |\ w(0,x)=0,\ w+m_0\geq 0 \},\\
&\widetilde\Bb_0:=\{v \in H^{2k-1}(\Omega_T)\ |\ v(T,x)=0 \}.
\end{align*}
 Consider the mapping $A: \widetilde\Aa_0 \times \widetilde\Bb_0 \to \widetilde\Aa_0 \times \widetilde\Bb_0$ defined, for   $(m_1, u_1) \in \widetilde \Aa_0
\times \widetilde\Bb_0$,  by
\begin{equation}\label{OpeA}
       A
    \begin{bmatrix}
      m_1  \\
      u_1        
    \end{bmatrix}
  :=
    \begin{bmatrix}
   m_1^\ast \\
   u_1^\ast
    \end{bmatrix},
\end{equation}
where $m_1^\ast \in \Aa_0$ is the unique minimizer to \eqref{VP1} and $u_1^\ast\in \Bb_0 $ is the unique solution to \eqref{BL1}.


\begin{pro}\label{ACC} 
Let  $H$, $g$, \(h\), $\sigma$,  $V$, \(\{a_{ij}\}_{i,j=1}^d\),
 \(m_0\), and $\xi$ be as in Problem~\ref{MP},
and assume that Assumption~\ref{hhyp}
holds. Then,  the mapping $A: \widetilde\Aa_0 \times \widetilde\Bb_0 \to \widetilde\Aa_0 \times \widetilde\Bb_0$ in \eqref{OpeA} is continuous and compact.
\end{pro}
%
%
\begin{proof}

We start by verifying the continuity of \(A\).
Let $(m_1, u_1),\, ( {m_1}_n, {u_1}_n) \in \widetilde\Aa_0  \times \widetilde\Bb_0 $ be such that 
${m_1}_n \to m_1$ in $H^{2k-2}(\Omega_T)$ and ${u_1}_n \to u_1$ in $H^{2k-1}(\Omega_T)$. We want to prove that 
${m_1}_n^\ast \to m_1^\ast$ in $H^{2k-2}(\Omega_T)$ and ${u_1}_n^\ast \to u_1^\ast$ in $H^{2k-1}(\Omega_T)$, where 
\begin{equation*}
\begin{aligned}
        \begin{bmatrix}
         m_1^\ast \\
         u_1^\ast
        \end{bmatrix}
= A \begin{bmatrix}
      m_1  \\
      u_1        
    \end{bmatrix}
 \enspace \text{and} \enspace 
 \begin{bmatrix}
         {m_1}_n^\ast \\
         {u_1}_n^\ast
        \end{bmatrix}
= A \begin{bmatrix}
      {m_1}_n  \\
      {u_1}_n        
    \end{bmatrix}. 
\end{aligned}
\end{equation*}

Recalling \eqref{defImu} and \eqref{defBfmu}, we define $I_n:=I_{({m_1}_n, {u_1}_n)}$ and $f_n := f_{({m_1}_n, {u_1}_n)}$. Because of the definition of \(A\), we have that 
$ (m_1^\ast,  u_1^\ast)$ and $ ({m_1}_n^\ast, {u_1}_n^\ast)$ belong to $  \Aa_0 \times \Bb_0$
 and satisfy, for all \(v\in \Bb_0\),
\begin{equation*}
\begin{aligned}
I_1[m_1^\ast]=\min_{w \in \Aa_0} I_1[w],\enspace I_n[{m_1}_n^\ast]=\min_{w \in \Aa_0} I_n[w],
\enspace
B[u_1^\ast,v] = \langle f_1,v \rangle,\enspace B[{u_1}_n^\ast,v] = \langle f_n,v \rangle.
\end{aligned}
\end{equation*}
Because  ${m_1}_n^\ast$ and $m_1^\ast$ are minimizers,  using the  second equality in
\eqref{eqvp-5}, we get
\begin{equation*}
\begin{aligned}
I_1[m_1^\ast]+I_n[{m_1}_n^\ast] &\leq I_1\left[\frac{m_1^\ast + {m_1}_n^\ast}{2}\right] + I_n\left[\frac{m_1^\ast
+ {m_1}_n^\ast}{2}\right] \\
&=\frac12 I_1[m_1^\ast] + \frac12 I_1[{m_1}_n^\ast] + \frac12 I_n[m_1^\ast] + \frac12 I_n[{m_1}_n^\ast]
\\&\quad\,-\int_0^T\!\!\!\int_{\Tt^d}\frac{\epsilon}{4}\Big[
\big(m_1^\ast-{m_1}_n^\ast\big)^2 +
\sum_{|\beta|=2k}\big(\partial_{t,x}^\beta
m_{1}^\ast - \partial_{t,x}^\beta {m_1}_n^\ast
\big)^2 \Big] \, \dx\dt, 
\end{aligned}
\end{equation*}
which can be rewritten as
\begin{equation*}\label{eq:bymin}
\begin{aligned}
&\int_0^T\!\!\!\int_{\Tt^d}\frac{\epsilon}{4}\Big[
\big(m_1^\ast-{m_1}_n^\ast\big)^2 +
\sum_{|\beta|=2k}\big(\partial_{t,x}^\beta
m_{1}^\ast - \partial_{t,x}^\beta {m_1}_n^\ast
\big)^2 \Big] \, \dx\dt\\&\quad \leq \frac12 \big(I_1[{m_1}_n^\ast]
+  I_n[m_1^\ast] - 
I_1[m_1^\ast]-  I_n[{m_1}_n^\ast]\big). \end{aligned}
\end{equation*}
%
Hence, Young's inequality yields
\begin{equation}\label{eq1-6-1}
 \begin{split}
  &\int_0^T\!\!\!\int_{\Tt^d}\frac{\epsilon}{4}\Big[ \big(m_1^\ast-{m_1}_n^\ast\big)^2 + \sum_{|\beta|=2k}\big(\partial_{t,x}^\beta m_{1}^\ast - \partial_{t,x}^\beta {m_1}_n^\ast \big)^2 \Big] \, \dx\dt \\
  \leq
  &\int_0^T\!\!\!\int_{\Tt^d}\frac{\epsilon}{8}(m_1^\ast
- {m_1}_n^\ast)^2 + \frac{1}{2\epsilon} \Big( | {u_1}_t - {{u_1}_n}_t |  + \sum_{i,j=1}^d|a_{ij}(x)| |{u_1}_{x_ix_j}-{{u_1}_n}_{x_ix_j}|\\
 & \quad +  \big| H(x,Du_1) - H(x,D{u_1}_n) \big|+ | \widehat g( m_1,\widehat h(\boldsymbol{m}_1))
- \widehat g({m_1}_n,\widehat h({\boldsymbol{m}_1}_n) ) | \Big)^2 \, \dx\dt. 
\end{split}
\end{equation}
From \eqref{eq:MorreyET}, there exists
a positive constant,
\(c>0\), independent of \(n\in\Nn\),
such that 
\begin{equation}
\label{eq:bounds}
\begin{aligned}
\sup_{n\in\Nn} 
\big(\Vert m_1\Vert_{L^{\infty}(\Omega_T)}
+ \Vert {m_1}_n\Vert_{L^{\infty}(\Omega_T)}
+\Vert u_1\Vert_{W^{1,\infty}(\Omega_T)}
+ \Vert {u_1}_n\Vert_{W^{1,\infty}(\Omega_T)}+
\Vert m_0\Vert_{L^{\infty}(\Omega_T)}
\big)
< c.
\end{aligned}
\end{equation}
Then, using \eqref{eq1-6-1}, \eqref{eq:bounds}, the facts that \(H\), \(D_pH\), and \(g\)
are locally Lipschitz functions, $\sigma$
and \(a_{ij}\) are bounded, and \eqref{hFD}--\eqref{hbounds}
hold with \(\kappa=2k-1\) an with \(m\) replaced by \( {m_1}_n+m_0\)
and \(\bar
m\) replaced by  \(m_1 + m_0\), we can find a positive constant,
\(C\), independent of \(n\in\Nn\), such
that %
\begin{equation*}
 \begin{split}
   &\int_0^T\!\!\!\int_{\Tt^d}\frac{\epsilon}{8}\Big[
\big(m_1^\ast-{m_1}_n^\ast\big)^2 +
\sum_{|\beta|=2k}\big(\partial_{t,x}^\beta
m_{1}^\ast - \partial_{t,x}^\beta {m_1}_n^\ast
\big)^2 \Big] \, \dx\dt\\&\quad\leq \frac{C}{\epsi}\Big(\Vert{m_1}_n - m_1\Vert_{H^{2k-2}(\Omega_T)}^2 +
\Vert{u_1}_n
- u_1\Vert_{H^{2k-1}(\Omega_T)}^2\Big).
\end{split}
\end{equation*}
Because ${m_1}_n \to m_1$ in $H^{2k-2}(\Omega_T)$ and ${u_1}_n \to u_1$ in $ H^{2k-1}(\Omega_T)$,  we have 
\begin{equation*}
\lim_{n\to\infty}\Vert m_1^\ast - {m_1}_n^\ast \Vert_{L^2(\Omega_T)} = 0, \quad  \lim_{n\to\infty}\sum_{|\beta|=2k}\Vert \partial_{t,x}^\beta m_{1}^\ast - \partial_{t,x}^\beta {m_1}_n^\ast\Vert_{L^2(\Omega_T)} = 0.
\end{equation*}
Then, invoking the Gagliardo--Nirenberg interpolation inequality, we obtain ${m_1}_n^\ast  \to m_1^\ast$ in $H^{2k}(\Omega_T)$,
and thus  in $H^{2k-2}(\Omega_T)$. 

Next, we prove that ${u_1}_n^\ast$ converges to $u_1^\ast$ in $H^{2k}(\Omega_T)$,
and thus in $H^{2k-1}(\Omega_T)$. Recalling  \eqref{defBfmu} and \eqref{eq:bounds}, similar arguments
to those above yield
\begin{equation}
\label{eq:eq2-6-1}
\begin{aligned}
& \,\epsilon C\Vert  
 u_{1}^\ast - u_{n}^\ast\Vert_{H^{2k}(\Omega_T)}^2
 \leq  B[u_1^\ast-{u_1}_n^\ast, u_1^\ast-{u_1}_n^\ast]
= 
\langle f_1-f_n, u_1^\ast-{u_1}_n^\ast \rangle \\
 \leq &\, \int_0^T\!\!\!\int_{\Tt^d} \bigg[|{{m_1}_n}_t - {m_1}_t ||u_1^\ast-{u_1}_n^\ast | 
 + \bigg|\sum_{i,j=1}^d\big(a_{ij}(x)({m_1}_n - m_1)\big)_{x_ix_j}\bigg||u_1^\ast-{u_1}_n^\ast | \\
 &\quad
 + \Big(|m_1 D_pH(x,Du_1)-{m_1}_n D_pH(x,Du_1)  + {m_1}_n D_pH(x,Du_1)- {m_1}_n D_pH(x,D{u_1}_n)| \\
 &\quad
 + \sigma | D_pH(x,Du_1) - D_pH(x,D{u_1}_n)| \Big)|Du_1^\ast-D{u_1}_n^\ast|\bigg]\,\dx\dt\\
 \leq  &\, \frac{\epsilon C}{2}\Vert 
 u_{1}^\ast - u_{n}^\ast\Vert_{H^{2k}(\Omega_T)}
 + \frac{\tilde C}{\epsi}\big(\Vert m_1 - {m_1}_n \Vert_{H^2(\Omega_T)}^2
+ \Vert D u_1 -D {u_1}_n \Vert_{L^2(\Omega_T)}^2\big)
\end{aligned}
\end{equation}
for some constants \(C, \, \tilde C>0\)  independent of \(n\in\Nn\). From \eqref{eq:eq2-6-1},
 we conclude that ${u_1}_n^\ast \to u_1^\ast$ in $H^{2k}(\Omega_T)$.

Finally, we prove the compactness of $A$. We want to show that if \(\{({m_1}_n,{u_1}_n)\}_{n=1}^\infty\) is a bounded sequence in 
$\widetilde\Aa_0 \times \widetilde\Bb_0$, then \(\{A({m_1}_n,{u_1}_n)\}_{n=1}\) is pre-compact  in 
$\widetilde\Aa_0 \times \widetilde\Bb_0$. This is a consequence of
\eqref{eq:MorreyET}, Assumption~\ref{hhyp}, Corollary~\ref{aprimvp1}, Proposition~\ref{ESBP1},
and the compact embedding  \(H^{2k}(\Omega_T) \times H^{2k}(\Omega_T) \hookrightarrow H^{2k-2}(\Omega_T)\times H^{2k-1}(\Omega_T)\) due to the Rellich--Kondrachov theorem. 
\end{proof}

As we noted before, applying Schaefer's fixed-point theorem, we verify the existence of weak solutions to Problem~\ref{MP}. We introduce next the precise version of this theorem that we use, see Theorem~6.2 in \cite{FGT1}.

\begin{theorem}\label{thm:SFPT} Let \(X\) be a convex and closed subset of a Banach space such that   \(0\in X\). Suppose that \(A:X \to X\) is a continuous and compact mapping such that the set 
\begin{equation*}
\begin{aligned}
\big\{w\in X  |\  w=\lambda A[w] \hbox{ for some } \lambda \in [0,1]\big\}
\end{aligned}
\end{equation*}
is bounded. Then, \(A\) has a fixed point.
\end{theorem}

\begin{pro}\label{ExiUniS1}
Consider Problem~\ref{MP}, let $A$ be the mapping defined in \eqref{OpeA}, and suppose that Assumptions~\ref{Hconv}--\ref{AWC1}
and \ref{Hmono} hold for some \(\gamma>1\). Then, there exists a unique weak solution, $(m,u) \in H^{2k}(\Omega_T) \times H^{2k}(\Omega_T)$, to Problem~\ref{MP} with $u_T=0$ in the sense of Definition~\ref{DMPWS1}.
\end{pro}
%
%
\begin{proof}
\textit{(Existence) }Fix $\lambda \in
[0,1]$, and let $(m_\lambda, u_\lambda)\in
\widetilde\Aa_0 \times \widetilde\Bb_0$
be such that
\begin{equation}\label{OLA1}
   \begin{bmatrix}
      m_\lambda \\
      u_\lambda 
    \end{bmatrix}
  =  \lambda A
    \begin{bmatrix}
      m_\lambda  \\
      u_\lambda
    \end{bmatrix}.
\end{equation}
If $\lambda=0$, then $(m_\lambda,u_\lambda)=(0,0)$.
Suppose that $0<\lambda \leq1$ and that there exists a pair $(m_\lambda, u_\lambda)$ satisfying \eqref{OLA1}; then,
because of the definition of \(A\), Proposition~\ref{EMVP1},
Corollary~\ref{PVI}, and Proposition~\ref{ESBP1},
we obtain \(\frac{m_\lambda}\lambda \in\Aa_0\),
\(\frac{u_\lambda}\lambda \in \Bb_0\),
and
\begin{equation*}
\begin{aligned}
&\int_0^T\!\!\!\int_{\Tt^d} \lambda \big( {u_\lambda}_t + \sum_{i,j=1}^da_{ij}(x){u_{\lambda}}_{x_ix_j}
- H(x,Du_\lambda) + \widehat g(m_\lambda,\widehat h(\boldsymbol{m}_{\lambda})) - V
\big)(\lambda w-m_\lambda) \, \dx\dt\\
&\quad+\int_0^T\!\!\!\int_{\Tt^d}  \epsilon\Big 
  (m_\lambda+\lambda\widehat\sigma +
  \lambda\sum_{|\beta|=2k}\partial_{t,x}^{2\beta}
  \widehat\sigma\Big)(\lambda w-m_\lambda)
  +
  \epsilon \sum_{|\beta|=2k}\partial_{t,x}^\beta
  m_\lambda (\lambda \partial_{t,x}^\beta w-\partial_{t,x}^\beta m_\lambda)\, \dx\dt
  \geq 0 
\end{aligned}
\end{equation*}
and
\begin{equation*}
\begin{aligned}
&\int_0^T\!\!\!\int_{\Tt^d} \Big[
   \lambda\Big({m_\lambda}_t - \sum_{i,j=1}^d\big(a_{ij}(x)(m_{\lambda}+\widehat\sigma)\big)_{x_ix_j} -\div\big((m_\lambda+\widehat\sigma )
D_p H(x,Du_\lambda)\big) \Big)v\,
\Big]\dx\dt\\ 
&\quad +\int_0^T\!\!\!\int_{\Tt^d}\bigg[
  \epsilon\Big( u_\lambda v + \sum_{|\beta|= 2k}
\partial_{t,x}^\beta u_\lambda\partial_{t,x}^\beta
v\Big)+\epsilon\lambda\Big(\xi +\sum_{|\beta|=2k}\partial_{t,x}^{2\beta}\xi\Big)v
\bigg]\, \dx\dt=0
\end{aligned}
\end{equation*}
for all \(w\in\Aa_0\) and \(v\in \Bb_0\).
Consequently, taking  \(w=0\)
and \(v= u_\lambda\) in these two conditions, and arguing as in  Proposition~\ref{apriWS1} using the Gagliardo--Nirenberg interpolation inequality  and the conditions \(m_\lambda + \widehat\sigma =m_\lambda +m_0+ \sigma   \geq0\),
 $m_{\lambda}(0,\cdot)=0$, and \(u_\lambda(T,\cdot)=0\), we get
\begin{equation}\label{E3.1}
\begin{aligned}
  &\int_0^T\!\!\!\int_{\Tt^d} \lambda \left[m_\lambda
\widehat g(m_\lambda, \widehat h(\boldsymbol{m}_\lambda))+ (m_\lambda+\widehat\sigma) |Du_\lambda|^\gamma \right]\,
\dx\dt
 \\&\quad+\epsilon \int_0^T\!\!\!\int_{\Tt^d} \Big[
m_\lambda^2 +u_\lambda^2 + \sum_{|\beta|=2k}(\partial_{t,x}^\beta
m_\lambda)^2+\sum_{|\beta|= 2k}(\partial_{t,x}^\beta u_{\lambda})^2
 \Big] \, \dx\dt
\leq C, 
\end{aligned}
\end{equation}
where $C$ is a positive constant independent
of $\lambda$. 

Next, we observe that because \(m_\lambda + m_0 \geq0\), we can
use  Assumptions~\ref{gint} and \ref{AWC1} with \(\delta
= \frac{1}{2(1+\Vert m_0\Vert_\infty)}\) to conclude that
\begin{equation*}
\begin{aligned}
&\int_0^T\!\!\!\int_{\Tt^d} m_\lambda
\widehat g(m_\lambda, \widehat h(\boldsymbol{m}_\lambda)) \, \dx\dt\\ = &\int_0^T\!\!\!\int_{\Tt^d} (m_\lambda + m_0)
 g(m_\lambda + m_0,  h(\boldsymbol{m}_\lambda
+ \boldsymbol{m}_0)) \,
\dx\dt
 - \int_0^T\!\!\!\int_{\Tt^d} m_0 
 g(m_\lambda + m_0,  h(\boldsymbol{m}_\lambda
+ \boldsymbol{m}_0)) \,
\dx\dt\\
\geq &\,\frac12  \int_0^T\!\!\!\int_{\Tt^d} (m_\lambda + m_0)
 g(m_\lambda + m_0,  h(\boldsymbol{m}_\lambda
+ \boldsymbol{m}_0)) \,
\dx\dt - C_\delta\Vert m_0\Vert_\infty \geq \frac{C}{2} - C_\delta\Vert m_0\Vert_\infty.
  \end{aligned}
\end{equation*}
This estimate, \eqref{E3.1}, and the condition \(m_\lambda + \widehat\sigma
\geq0\) yield

\begin{equation*}
\begin{aligned}
\epsilon \int_0^T\!\!\!\int_{\Tt^d} \Big[
m_\lambda^2+u_\lambda^2 +\sum_{|\beta|=2k}(\partial_{t,x}^\beta
m_\lambda)^2 + \sum_{|\beta|= 2k}(\partial_{t,x}^\beta u_{\lambda})^2
 \Big] \, \dx\dt
\leq C,
\end{aligned}
\end{equation*}
where $C$ is another constant independent of $\lambda$.
Invoking the Gagliardo--Nirenberg interpolation
inequality,  we verify that $(m_\lambda,
u_\lambda)$ is uniformly bounded in
$H^{2k}(\Omega_T) \times H^{2k}(\Omega_T)$
with respect to $\lambda$. From this fact and Proposition~\ref{ACC}, we can use Theorem~\ref{thm:SFPT} and conclude
that \(A \) has a fixed point, \((\widetilde m,u)\in
\widetilde \Aa_0 \times \widetilde\Bb_0\). 

Let $\bar m:=\widetilde m+m_0$. By the definition
of \(A\),  
Proposition~\ref{EMVP1},
Proposition~\ref{PVI}, and Proposition~\ref{ESBP1},
we conclude that \((\bar m,u) \in \Aa \times
\Bb_0\) and
\begin{itemize}
\item[(i)] \eqref{VI1} holds with \((m_1,u_1)\)
replaced by \((\widetilde m, u)\) and \(m\)
replaced by \(\widetilde m\),

\item[(ii)] \eqref{BL1} holds with \(f_1\)
replaced by \(f_{(\widetilde m, u)}\).
\end{itemize} 
Recalling that \(\widehat \sigma = \sigma+m_0\)
and \(\widehat g(\widetilde m,\widehat
h(\boldsymbol{\widetilde m})) = g(\widetilde m + m_0,
h(\boldsymbol{\widetilde m} +\boldsymbol{ m_0})) =  g(\bar m,
h(\boldsymbol{\bar m}))\), condition
(i) becomes
\begin{equation*}
\begin{aligned}
&
 \int_0^T\!\!\!\int_{\Tt^d} \bigg( 
 u_t + \sum_{i,j=1}^da_{ij}(x)u_{x_ix_j}
- H(x,Du) +g(\bar m,
h(\boldsymbol{\bar m}))+
V(t,x) \bigg)(w+m_0-\bar m) \, \dx\dt \\ 
&\quad +\int_0^T\!\!\!\int_{\Tt^d} \bigg[
\epsilon\bigg (\bar m+\sigma+\sum_{|\beta|=2k}
\partial^{2\beta}_{t,x} \sigma\bigg)(w+m_0-\bar
m)
+\epsilon \sum_{|\beta|=2k}\partial_{t,x}^\beta
\bar m\big(\partial_{t,x}^\beta (w+m_0)-\partial_{t,x}^\beta
\bar m\big)\bigg]\, \dx\dt \geq 0
\end{aligned}
\end{equation*}
for all \(w\in \Aa_0\). Observing that
if \(\bar w\in \Aa\), then \(w:= \bar
w - m_0\in \Aa_0\), we conclude from
the previous estimate  that  condition 
(E2) in Definition~\ref{DMPWS1} holds
for \((\bar m, u)\). Moreover, condition
(ii) above is equivalent to condition
(E3) in Definition~\ref{DMPWS1} 
for \((\bar m, u)\). Consequently,  
$(\bar m,u)$ belongs to $  H^{2k}(\Omega_T)
\times H^{2k}(\Omega_T)$ and satisfies
(E1)--(E3) in Definition~\ref{DMPWS1}
 with \(u_T=0\).
\smallskip

\textit{(Uniqueness) }
Suppose that   $(m_1,u_1)$ and $(m_2,u_2)$
two weak solutions to Problem~\ref{MP} with $u_T=0$ in
the sense of Definition~\ref{DMPWS1}.
Choosing $w=m_2$ for \((u_1,m_1)\) and $w=m_1$ for \((u_2,m_2)\) in (E2) of Definition~\ref{DMPWS1}, and then adding the resulting inequalities, we have
\begin{equation}\label{ineqUni-1}
\begin{aligned}
&\int_0^T\!\!\!\int_{\Tt^d}\Big[ - ( {u_1}_t-{u_2}_t ) - \sum_{i,j=1}^da_{ij}(x)( {u_{1}}_{x_ix_j}-{u_{2}}_{x_ix_j}) + H(x,Du_1)-H(x,Du_2)\\ 
&\quad- ( g(m_1, h(\boldsymbol{m}_1)) -  g(m_2, h(\boldsymbol{m}_2)) ) \Big](m_1-m_2)\,\dx\dt\\
&\quad - \int_0^T\!\!\!\int_{\Tt^d} \Big[\epsilon(m_1-m_2)^2+\epsilon\sum_{|\beta|=2k}(\partial_{t,x}^{\beta}m_1 - \partial_{t,x}^{\beta} m_2)^2\Big]\, \dx\dt
\geq0.
\end{aligned}
\end{equation}
Because $u_1-u_2\in \Bb_0$, setting $v=u_1-u_2$ in (E3) of Definition~\ref{DMPWS1}
for  \((u_1,m_1)\) and \((u_2,m_2)\), and then subtracting the resulting equalities, we have
\begin{equation}\label{ineqUni-2}
\begin{aligned}
&\int_0^T\!\!\!\int_{\Tt^d}\Big({m_1}_t-{m_2}_t-\sum_{i,j=1}^d\big(a_{ij}(x)(m_1-m_2)\big)_{x_ix_j}\\
&\quad
-\div\big((m_1+ \sigma) D_pH(x,Du_1)
- (m_2+\sigma) D_pH(x,Du_2)\big)\Big)(u_1-u_2)\,\dx\dt\\
&\quad
+\epsilon\int_0^T\!\!\!\int_{\Tt^d}
\Big((u_1 - u_2)^2 + \sum_{|\beta|=2k}(\partial_{t,x}^\beta u_1-\partial_{t,x}^\beta u_2)^2\Big)\, \dx\dt=0.
\end{aligned}
\end{equation}
Subtracting \eqref{ineqUni-1} from \eqref{ineqUni-2},
 we obtain
\begin{equation*}
\begin{aligned}
0\geq &
\int_0^T\!\!\!\int_{\Tt^d}\Big[
\epsilon(m_2-m_2)^2+\epsi(u_1 - u_2)^2+\epsilon\sum_{|\beta|=2k}(\partial_{t,x}^\beta m_2-\partial_{t,x}^\beta m_2)^2
+\epsilon\sum_{|\beta|=2k}(\partial_{t,x}^\beta u_1-\partial_{t,x}^\beta u_2)^2
\Big]\,\dx\dt\\
& + \int_0^T\!\!\!\int_{\Tt^d} \sigma(D_p H(x, Du_1)-D_p H(x,Du_2))\cdot (Du_1-Du_2) \,\dx\dt\\
&+\left\langle
   F
  \begin{bmatrix}
      m_2  \\
      u_2 
  \end{bmatrix}
  -
   F
  \begin{bmatrix}
      m_1  \\
      u_1 
  \end{bmatrix},
  \begin{bmatrix}
      m_2  \\
      u_2
  \end{bmatrix}
  -
  \begin{bmatrix}
      m_1  \\
      u_1
  \end{bmatrix}
\right\rangle
\geq 0
\end{aligned}
\end{equation*}
because each of the three terms in preceding
sum is nonnegative by Remark~\ref{Hconv2}, the positivity
of \(\sigma\), and  Assumption~\ref{Hmono}.  Then, each of these three terms must
 be equal to zero,  from which we conclude that $(m_1,u_2)=(m_2,u_2)$.
\end{proof}

\begin{proof}[Proof of Theorem~\ref{MT}] 
For \((t,x) \in\Omega_T\) and \(p\in\Rr^d\), define $\widehat H(x,p):=H(x,p+Du_T(x))$, $\widehat{V}(t,x):= V(t,x) + \sum_{i,j=1}^da_{ij}(x)
{u_{T}}_{x_ix_j}(x)$,  and 
$\widehat{\xi} (t,x):= \xi(t,x) +u_T(x)$.

Note that Assumptions~\ref{Hconv}--\ref{AWC1} and \ref{Hmono} also hold with \(H\), \(V\), and \(\xi\) replaced by
\(\widehat H\), \(\widehat
V\), and \(\widehat \xi\), respectively, for the same \(\gamma>1\) and possibly
different constants. 
Moreover, \((u,m)\in H^{2k}(\Omega_T)\times H^{2k}(\Omega_T)\) satisfies (E1)--(E3)
if and only if \((m,\bar u):=(m,u -u_T)\in H^{2k}(\Omega_T)
\times H^{2k}(\Omega_T)\) satisfies (E1)--(E3) with \(u_T=0\) and with \(H\), \(V\),  and \(\xi\) replaced by \(\widehat H\),  \(\widehat V\), and \(\widehat \xi\), respectively. 

To conclude the proof, we use Proposition~\ref{ExiUniS1} that shows that there exists a unique pair  \((m,\hat u)\in H^{2k}(\Omega_T)
\times H^{2k}(\Omega_T)\) satisfying (E1)--(E3) with \(u_T=0\) and
 with \(H\), \(V\),  and \(\xi\) replaced by \(\widehat H\),
 \(\widehat V\), and \(\widehat \xi\), respectively.  
\end{proof}

\section{Proof of Theorem~\ref{TOP}}\label{PfMT2}
To prove Theorem~\ref{TOP}, we begin by investigating the compactness, with respect to \(\epsi\), of weak solutions to Problem~\ref{MP}. Then, we define a linear functional, $F_\epsilon$, associated with \eqref{HighdRP}, and we show that $F_\epsilon$ is monotone. Next, by Minty's method, we prove Theorem~\ref{TOP}.
Moreover, we study consistency of weak solutions. In particular, if a weak solution $(m,u)$ has enough regularity and $m\mapsto g(m,h(\boldsymbol{m}))$ is strictly monotone with respect to the $ L^2$-inner product, then we show that the weak solution is the unique classical solution to Problem~\ref{OP}.

Set $\sigma\equiv 0$ and $\xi\equiv 0$. Let $(m_\epsilon, u_\epsilon)$ be the weak solution given by Theorem~\ref{MT}. Then, we define $\big<u_\epsilon\big>:t\to \Rr$ and $\tilde u_\epsilon:\Omega_T\to \Rr$ by
$\big<u_\epsilon\big>:=\int_{\Tt^d}u_\epsilon(t,x)\,\dx$ and
\begin{equation}
\tilde u_\epsilon(t,x):=u_\epsilon(t,x)-\big<u_\epsilon\big>(t).\label{Defutild} 
\end{equation}
The next lemma addresses the weak convergence of $(m_\epsilon,\tilde u_\epsilon)$ in $L^1(\Omega_T)\times L^\gamma((0,T);W^{1,\gamma}(\Tt^d))$.
\begin{lem}\label{WCWS1}
Consider Problem~\ref{MP} with $\sigma= 0$ and $\xi= 0$, and suppose that Assumptions~\ref{Hconv}--\ref{Hmono}  hold for some \(\gamma>1\). Let $(m_\epsilon, u_\epsilon) \in H^{2k}(\Omega_T) \times H^{2k}(\Omega_T)$ be the unique weak solution to Problem~\ref{MP} and let $\tilde u_\epsilon$ be given by \eqref{Defutild}. Then, there exists $(m,\tilde u)\in L^1(\Omega_T) \times L^\gamma((0,T);W^{1,\gamma}(\Tt^d))$ such that $m \geq 0$
and $(m_\epsilon,\tilde u_\epsilon)$ converges to $(m,\tilde u)$ weakly in $L^1(\Omega_T)\times L^\gamma((0,T);W^{1,\gamma}(\Tt^d))$ as $\epsilon \to 0$, extracting a subsequence if necessary. 
\end{lem}
\begin{proof}
Because $\big<u_\epsilon\big>(t)=\int_{\Tt^d} \tilde u_\epsilon(t,x)\,\dx=0$ for a.e.\,$t\in(0,T)$ and $\partial_x^\alpha\tilde u_\epsilon=\partial_x^\alpha u_\epsilon$, from Corollary~\ref{apriDu1} and Poincar\'{e}--Wirtinger inequality, we have
\begin{equation*}
\|\tilde u_\epsilon\|_{L^\gamma(\Omega_T)}^\gamma
=\int_0^T\!\!\!\int_{\Tt^d}|\tilde u_\epsilon -\big<u_\epsilon\big>|^\gamma\,\dx\dt
\leq
C \|D\tilde u_\epsilon\|_{L^\gamma(\Omega_T)}^\gamma
\leq C,
\end{equation*}
where $C$ is independent of $\epsilon$. Because $L^\gamma((0,T);W^{1,\gamma}(\Tt^d))$ is a reflexive Banach space, there exist a subsequence $\{\tilde u_{\epsilon_j}\}_{j=1}^\infty$ and $\tilde u\in L^\gamma((0,T);W^{1,\gamma}(\Tt^d))$ such that $\tilde u_{\epsilon_j}\rightharpoonup \tilde u$ in 
$L^\gamma((0,T);W^{1,\gamma}(\Tt^d))$.

On the other hand, using Proposition~\ref{apriWS1}
and the positivity
of \(m_\epsilon\) and \(m_0\), we get
\begin{equation*}
\sup_{\epsilon\in(0,1)}
\int_0^T\!\!\!\int_{\Tt^d} m_\epsilon g(m_\epsilon,h(\boldsymbol{m}_\epsilon))\, \dx\dt <\infty.
\end{equation*}
Therefore, from Assumption~\ref{gwc}, there exists $m\in L^1(\Omega_T)$ such that $m_\epsilon \rightharpoonup m$ in $L^1(\Omega_T)$ as $\epsilon \to 0$, extracting a subsequence if necessary. Since $m_\epsilon\geq0$, we conclude that $m\geq 0$.
\end{proof}

Fix $(\eta, v) \in H^{2k}(\Omega_T) \times H^{2k}(\Omega_T)$, let \(F[\eta,v]\) be the functional introduced in \eqref{DRF2},  and let 
 $F_{\epsilon} [\eta,v]: H^{2k}(\Omega_T) \times H^{2k}(\Omega_T) \to \Rr$ be the linear functional given by
\begin{equation}\label{DRF1}
\begin{aligned}
\left\langle
   F_\epsilon
  \begin{bmatrix}
      \eta \\
      v
  \end{bmatrix},
  \begin{bmatrix}
      w_1 \\
      w_2
  \end{bmatrix}
\right\rangle
:= & \left\langle
   F
  \begin{bmatrix}
      \eta \\
      v
  \end{bmatrix},
  \begin{bmatrix}
      w_1 \\
      w_2
  \end{bmatrix}
\right\rangle +\epsi\int_0^T\!\!\!\int_{\Tt^d}\Big( \eta w_1 +  \sum_{|\beta|=2k}\partial_{t,x}^\beta \eta \partial_{t,x}^\beta w_1 \Big)\,\dx\dt
\\
& + \epsilon\int_0^T\!\!\!\int_{\Tt^d} \Big( v w_2 + \sum_{|\beta|=2k}\partial_{t,x}^\beta v\partial_{t,x}^\beta w_2 \Big)\,\dx\dt.
\end{aligned}
\end{equation}
Next, we prove the monotonicity of $F_\epsilon$ over $\Aa\times \Bb$, where \(\Aa\) and 
$\Bb$ are given by \eqref{DAa} and \eqref{DBb}.

%
%
\begin{lem}\label{PAm}
Consider Problem~\ref{MP} with $\sigma= 0$ and $\xi= 0$, and suppose that Assumptions~\ref{Hconv}--\ref{Hmono}  hold for some \(\gamma>1\). Let $(m_\epsilon, u_\epsilon) \in H^{2k}(\Omega_T) \times H^{2k}(\Omega_T)$ be the unique weak solution to Problem~\ref{MP} and let $m$ be given by Lemma~\ref{WCWS1}. Then, we have
\begin{equation*}
\frac{d}{dt}\left(\int_{\Tt^d}m(t,x)\,\dx\right)=0 \ \ \mbox{in the sense of distributions};
\end{equation*}
that is, for $v\in C_c^\infty(0,T)$,
\begin{equation}\label{Dst1}
\int_0^Tv'(t)\bigg(\int_{\Tt^d}m(t,x)\,\dx\bigg)\dt=0.
\end{equation}
\end{lem}
%
\begin{proof}
Fix $v\in C_c^{\infty}(0,T)$. Because $v \in \Bb_0$ and $v$ is independent of the space variables, from (E3) with $\sigma=0$ and $\xi= 0$, and using  integration by parts, we have
\begin{equation*}
-\int_0^T\!\!\!\int_{\Tt^d} m_\epsilon v_t \,\dx\dt
+
\epsilon\int_0^T\!\!\!\int_{\Tt^d}  \Big( u_\epsi v + \sum_{|\beta|=2k}\partial_{t,x}^\beta
u_\epsi \partial_{t,x}^\beta v \Big)\,\dx\dt
=0.
\end{equation*}
Thus, by Corollary~\ref{aprisqrt} and Lemma~\ref{WCWS1}, as $\epsilon\to 0$, we have that \eqref{Dst1} holds.
\end{proof}

%
%
\begin{lem}\label{MORF}
Let $H$, $g$, $h$, $V$, and \(\{a_{ij}\}_{i,j=1}^d\) be as in Problem~\ref{MP}, let $F_\epsilon$ be given by \eqref{DRF1}, and suppose that Assumption~\ref{Hmono} holds. 
Then, for any $(\eta_1, v_1)$, $(\eta_2,  v_2) \in \Aa\times 
\Bb$, we have
\begin{equation*}
\left\langle
  F_{\epsilon}
  \begin{bmatrix}
      \eta_1  \\
      v_1 
  \end{bmatrix}
     -
      F_{\epsilon}
      \begin{bmatrix}
          \eta_2  \\
          v_2 
      \end{bmatrix},
      \begin{bmatrix}
          \eta_1  \\
          v_1
      \end{bmatrix}
    -
   \begin{bmatrix}
         \eta_2  \\
         v_2
  \end{bmatrix}
\right\rangle
\geq 0.
\end{equation*}
\end{lem}
%
\begin{proof}
Let $(\eta_1, v_1)$, $(\eta_2,  v_2) \in \Aa\times 
\Bb$. Then, $v_1 - v_2 \in \Bb_0$. Hence, from Assumption~\ref{Hmono} and integration by parts, we have
\begin{align*}
&\left\langle
  F_{\epsilon}
  \begin{bmatrix}
      \eta_1  \\
      v_1 
  \end{bmatrix}
  -
  F_{\epsilon}
  \begin{bmatrix}
      \eta_2  \\
      v_2 
  \end{bmatrix},
  \begin{bmatrix}
      \eta_1  \\
      v_1
  \end{bmatrix}
  -
  \begin{bmatrix}
      \eta_2  \\
      v_2
  \end{bmatrix}
\right\rangle\\
&\quad\geq\int_0^T\!\!\!\int_{\Tt^d}
\epsilon\Big[(\eta_1-\eta_2)^2 + \sum_{|\beta|=2k}(\partial_{t,x}^\beta \eta_1-\partial_{t,x}^\beta \eta_2)^2  + (v_1 - v_2)^2+ \sum_{|\beta|= 2k}(\partial_{t,x}^\beta v_1-\partial_{t,x}^\beta v_2)^2\Big]\, \dx\dt
\geq 0. \qedhere
\end{align*}
\end{proof}
%
\begin{remark}\label{rmmono1}
Let $\widehat \Aa$ be the set in  \eqref{DAawh}. Then, Lemma~\ref{MORF} still holds for $(\eta_1,v_1), (\eta_2,v_2)\in\widehat \Aa\times \Bb$.
\end{remark}

\begin{lem}\label{Equtu}
Let $H$, $g$, $h$, $V$, and \(\{a_{ij}\}_{i,j=1}^d\) be  as in Problem~\ref{MP}, let $F_\epsilon$ be given by \eqref{DRF1}, and suppose that Assumptions \ref{Hconv}--\ref{AWC1}
and \ref{Hmono} hold.
Let $(m_\epsilon,u_\epsilon)\in H^{2k}(\Omega_T)\times H^{2k}(\Omega_T)$ be the unique weak solution to Problem~\ref{MP} and let $\tilde u_\epsilon$ be as in \eqref{Defutild}.
Then, for $(\eta, v)\in \widehat\Aa\times \Bb$, we have
\begin{equation*}
\begin{aligned}
\left\langle
  F
  \begin{bmatrix}
      \eta \\
      v
  \end{bmatrix},
  \begin{bmatrix}
      \eta  \\
      v 
  \end{bmatrix}
  -
  \begin{bmatrix}
      m_\epsilon \\
      u_\epsilon
  \end{bmatrix}
\right\rangle 
=
\left\langle
  F
  \begin{bmatrix}
      \eta \\
      v
  \end{bmatrix},
  \begin{bmatrix}
      \eta  \\
      v 
  \end{bmatrix}
  -
  \begin{bmatrix}
      m_\epsilon \\
      \tilde u_\epsilon
  \end{bmatrix}
\right\rangle .
\end{aligned}
\end{equation*}
\end{lem}
\begin{proof}
Using \eqref{Defutild}, the second term on the right-hand side of \eqref{DRF2} with \(w_2=v- u_\epsi\) can be written as
\begin{equation*}
\begin{split}
&\int_0^T\!\!\!\int_{\Tt^d}
\Big(
\eta_t-\sum_{i,j=1}^d(a_{ij}(x)\eta)_{x_ix_j}-\div\big(\eta D_p H(x,Dv)\big)\Big)( v- u_\epsilon )\,\dx\dt\\
=&
\int_0^T\!\!\!\int_{\Tt^d}
\Big(
\eta_t - \sum_{i,j=1}^d(a_{ij}(x)\eta)_{x_ix_j} 
- \div\big(\eta D_p H(x,Dv)\big)\Big)( v-\tilde u_\epsilon )\,\dx\dt\\
&-
\int_0^T\!\!\!\int_{\Tt^d}
\Big(
\eta_t - \sum_{i,j=1}^d(a_{ij}(x)\eta)_{x_ix_j} 
-\div\big(\eta D_p H(x,Dv)\big)\Big)\big<u_\epsilon\big>(t)\,\dx\dt.
\end{split}
\end{equation*}
Because \(\eta\in \widehat \Aa\), we have $\int_{\Tt^d} \eta_t(x,t)\,\dx = \frac{d}{dt}\left(\int_{\Tt^d}\eta(t,x)\,\dx\right)=0$
for a.e.~$t\in (0,T)$. Moreover,  using integration by parts with respect to the space variables together with the fact that  $\big<u_\epsilon\big>$ depends only on $t$, we conclude that
\begin{equation*}
\int_0^T\!\!\!\int_{\Tt^d}\Big(
\eta_t - \sum_{i,j=1}^d(a_{ij}(x)\eta)_{x_ix_j} 
-\div\big(\eta D_p H(x,Dv)\big)\Big)\big<u_\epsilon\big>(t)\,\dx\dt=0,
\end{equation*}
from which the statement follows.
\end{proof}

\begin{proof}[Proof of Theorem~\ref{TOP}]
Let $(m_\epsilon, u_\epsilon)\in H^{2k}(\Omega_T) \times H^{2k}(\Omega_T)$ be the unique weak solution to Problem~\ref{MP} in the sense of Definition~\ref{DMPWS1} and with $\sigma= 0$ and $\xi= 0$. Fix  $(\eta, v)\in \widehat\Aa \times \Bb$. By (E2) and (E3) in Definition~\ref{DMPWS1},  we have
\begin{equation*}
\left\langle
  F_\epsilon
  \begin{bmatrix}
      m_\epsilon \\
      u_\epsilon
  \end{bmatrix},
  \begin{bmatrix}
      \eta  \\
      v 
  \end{bmatrix}
  -
  \begin{bmatrix}
      m_\epsilon \\
      u_\epsilon
  \end{bmatrix}
\right\rangle
\geq  0.
\end{equation*}
Then, using Lemma~\ref{MORF}, we have
\begin{equation}\label{eq:pfe}
\begin{aligned}
0&\leq \left\langle
  F_{\epsilon}
  \begin{bmatrix}
      \eta \\
      v
  \end{bmatrix}-F_{\epsilon} 
  \begin{bmatrix}
      m_\epsilon \\
      u_\epsilon
  \end{bmatrix},
  \begin{bmatrix}
      \eta  \\
      v 
  \end{bmatrix}
  -
  \begin{bmatrix}
      m_\epsilon \\
      u_\epsilon
  \end{bmatrix}
\right\rangle \leq \left\langle
  F_{\epsilon}
  \begin{bmatrix}
      \eta \\
      v
  \end{bmatrix},
  \begin{bmatrix}
      \eta  \\
      v 
  \end{bmatrix}
  -
  \begin{bmatrix}
      m_\epsilon \\
      u_\epsilon
  \end{bmatrix}
\right\rangle= \left\langle
  F
  \begin{bmatrix}
      \eta \\
      v
  \end{bmatrix},
  \begin{bmatrix}
      \eta  \\
      v 
  \end{bmatrix}
  -
  \begin{bmatrix}
      m_\epsilon \\
      u_\epsilon
  \end{bmatrix}
\right\rangle + c_\epsi,
\end{aligned}
\end{equation}
where
\begin{equation}\label{vanish1}
\begin{aligned}
c_\epsi :=&\,\,\epsi\int_0^T\!\!\!\int_{\Tt^d}\Big( \eta (\eta- m_\epsi) +  \sum_{|\beta|=2k}\partial_{t,x}^\beta
\eta \partial_{t,x}^\beta  (\eta- m_\epsi) \Big)\,\dx\dt\\&\quad+ \epsilon\int_0^T\!\!\!\int_{\Tt^d} \Big( v(v- u_\epsi)+ \sum_{|\beta|=2k}\partial_{t,x}^\beta
v\partial_{t,x}^\beta (v- u_\epsi) \Big)\,\dx\dt. 
\end{aligned}
\end{equation}

Combining H\"older's inequality with Corollary~\ref{aprisqrt}, we get
\begin{equation}
\begin{aligned}\label{eq:lim1}
\lim_{\epsi\to0} c_\epsi = 0.
\end{aligned}
\end{equation}

On the other hand, by Lemma~\ref{WCWS1},  there exists $(m,\tilde u)\in L^1(\Omega_T) \times
L^\gamma((0,T);W^{1,\gamma}(\Tt^d))$  satisfying  (D1) in Definition~\ref{DOPWS1} and such that $(m_\epsilon, \tilde u_\epsilon)$ converges to $(m,\tilde u)$
weakly in $L^1(\Omega_T)\times L^\gamma((0,T);\allowbreak W^{1,\gamma}(\Tt^d))$ as $\epsilon \to 0$, extracting a subsequence if necessary.
Then,
by the definition of \(F[\eta,v]\)
(see \eqref{DRF2}) and Lemma~\ref{Equtu}, we obtain
\begin{equation}
\begin{aligned}
\lim_{\epsi\to0}  \left\langle
  F
  \begin{bmatrix}
      \eta \\
      v
  \end{bmatrix},
  \begin{bmatrix}
      \eta  \\
      v 
  \end{bmatrix}
  -
  \begin{bmatrix}
      m_\epsilon \\
      u_\epsilon
  \end{bmatrix}
\right\rangle 
\end{aligned}
=
\label{eq:lim2}
\begin{aligned}
\lim_{\epsi\to0}  \left\langle
  F
  \begin{bmatrix}
      \eta \\
      v
  \end{bmatrix},
  \begin{bmatrix}
      \eta  \\
      v 
  \end{bmatrix}
  -
  \begin{bmatrix}
      m_\epsilon \\
      \tilde u_\epsilon
  \end{bmatrix}
\right\rangle  = \left\langle
  F
  \begin{bmatrix}
      \eta \\
      v
  \end{bmatrix},
  \begin{bmatrix}
      \eta  \\
      v 
  \end{bmatrix}
  -
  \begin{bmatrix}
      m \\
      \tilde u
  \end{bmatrix}
\right\rangle.
\end{aligned}
\end{equation}

From \eqref{eq:pfe}, \eqref{eq:lim1}, and \eqref{eq:lim2}, we conclude that
\begin{equation*}
\begin{aligned}
 \left\langle
  F
  \begin{bmatrix}
      \eta \\
      v
  \end{bmatrix},
  \begin{bmatrix}
      \eta  \\
      v 
  \end{bmatrix}
  -
  \begin{bmatrix}
      m \\
      \tilde u
  \end{bmatrix}
\right\rangle\geq 0;
\end{aligned}
\end{equation*}
thus, \((m,\tilde u)\) also satisfies (D2) in Definition~\ref{DOPWS1}. Therefore, \((m,\tilde u)\) is a weak solution of Problem~\ref{OP} in the sense of Definition~\ref{DOPWS1}.
\end{proof}

%
\begin{lem}\label{Elam1}
Assume that Assumptions~\ref{Hconv}--\ref{Hmono} hold for some \(\gamma>1\), and let $(m,\tilde u)\in L^1(\Omega_T) \times L^{\gamma}((0,T);W^{1,\gamma}(\Tt^d))$ be a weak solution to Problem~\ref{OP} in the sense of Definition~\ref{DOPWS1}. Assume further   that 
  $(m , \tilde u)\in H^{2k}(\Omega_T)\times H^{2k}(\Omega_T)$ and that  $m(0,\cdot)=m_0(\cdot)$, $m>0$ in $\Omega_T$, and $\tilde u(T,\cdot)=u_T(\cdot)$. Then, there exists $\mu\in C(0,T)$ such that
\begin{equation}\label{Meq2}
{\tilde u}_t + \sum_{i,j=1}^da_{ij}(x){\tilde u}_{x_ix_j} 
- H(x,D\tilde u) + g(m,h(\boldsymbol{m})) + V(t,x) = \mu(t)\quad \hbox{ in } \Omega_T.
\end{equation}
\end{lem}
\begin{proof}
Recalling \eqref{eq:MorreyET}, since $m\in H^{2k}(\Omega_T)$ and thus, by Morrey's theorem, $m \in C^{1,l}(\Omega_T)$ for some $l\in (0,1)$, Lemma~\ref{PAm} yield
\begin{equation*}
\int_{\Tt^d}m(t,x)\,\dx = \int_{\Tt^d}m_0(x)\,\dx=1
\end{equation*}
for all  $t\in(0,T)$. Hence, $m\in \widehat\Aa$. 
Fix $\varphi\in C_c^\infty(\Omega_T)$ and $i\in\{1,...,d\}$. Then, $\int_{\Tt^d} \varphi_{x_i}(t,x)\,\dx=0$ for all $t\in(0,T)$. Thus, since $m>0$, we have  $m+\delta \varphi_{x_i}\in
\widehat \Aa$ for all  $\delta\in \Rr$  with \(|\delta|\) sufficiently small. Taking $\eta=m+\delta \varphi_{x_i}$ and $v=\tilde u$ in (D2) and integrating by parts on \(\Tt^d\), we obtain
\begin{equation}\label{Meq1}
\delta \int_0^T\!\!\!\int_{\Tt^d} \frac{\partial}{\partial x_i}\Big(
{\tilde u}_t + \sum_{i,j=1}^da_{ij}(x){\tilde u}_{x_ix_j} 
 - H(x,D\tilde u) + g(m+\delta\varphi, h(\boldsymbol{m}+\delta\boldsymbol{\varphi})) +V(t,x) \Big) \varphi\,\dx\dt
\geq 0.
\end{equation}
As the sign of $\delta $ is arbitrary, we conclude that \eqref{Meq1} holds with ``$\geq$" replaced by ``$=$". Then, dividing by \(\delta\in\Rr\setminus\{0\}\) first and letting  $\delta\to 0$ afterwards, we get  
\begin{equation*}
\int_0^T\!\!\!\int_{\Tt^d} \frac{\partial}{\partial x_i}\Big(
{\tilde u}_t + \sum_{i,j=1}^da_{ij}(x){\tilde u}_{x_ix_j} 
 - H(x,D\tilde u) + g(m,h(\boldsymbol{m})) +V(t,x) \Big) \varphi\,\dx\dt =0
\end{equation*}
by Lebesgue dominated convergence theorem (recall \eqref{hFD}).  Because $\varphi\in C_c^\infty(\Omega_T)$ and and $i\in\{1,...,d\}$ are arbitrary, the preceding equality implies that for each $t\in (0,T)$, there exists $\mu(t)$ for which \eqref{Meq2} holds a.e.~in \(\Omega_T\). 
Since the left-hand side of \eqref{Meq2} belongs to $C(\Omega_T)$,  we conclude that $\mu\in C(0,T)$ and \eqref{Meq2} holds pointwise in \(\Omega_T\).
\end{proof}

\begin{pro}\label{ECS1}
Assume that Assumptions~\ref{Hconv}--\ref{Hmono} hold for some \(\gamma>1\), and let $(m,\tilde u)\in L^1(\Omega_T) \times L^{\gamma}((0,T);W^{1,\gamma}(\Tt^d))$ be a weak solution to Problem~\ref{OP} in the sense of Definition~\ref{DOPWS1}. 
Assume further that $(m , \tilde u)\in H^{2k}(\Omega_T)\times H^{2k}(\Omega_T)$ satisfies  $m(0,\cdot)=m_0(\cdot)$, $m>0$ in $\Omega_T$, and $\tilde u(T,\cdot)=u_T(\cdot)$
and that  
$m\mapsto g(m,h(\boldsymbol{m}))$ is a  strictly monotone map with respect to the $ L^2$-inner product; that is, for $m_1, m_2 \in L^2(\Omega_T)$ with $m_1\neq m_2$, we have
\begin{equation}\label{stmono1}
\int_{0}^{T}\!\!\!\int_{\Tt^d}
(g(m_1,h(\boldsymbol{m_1}))-g(m_2,h(\boldsymbol{m_2})))(m_1-m_2)
\,\dx\dt>0.
\end{equation} 
Let $\mu\in C(\Omega_T)$ be given by Lemma~\ref{Elam1}, and define  $u:=\tilde u +\int_t^T \mu(s)\, \ds$. Then, $(m, u)$ is the unique classical solution to Problem~\ref{OP}.
\end{pro}
\begin{proof}

We have \(u_t (x,t) ={\tilde u}_t(x,t) -\mu(t) \) and 
\(\partial^\alpha_x  u (x,t) =\partial^\alpha_x {\tilde u} (x,t)\) for all \((x,t) \in \Omega_T\). Because \((m,\tilde u)\) satisfies \eqref{Meq2} pointwise in \(\Omega_T\), we conclude that \((m,u)\) satisfies 
\begin{align*}
&{u}_t + \sum_{i,j=1}^da_{ij}(x) u_{x_ix_j} 
  - H(x,Du) + g(m,h(\boldsymbol{m})) + V(t,x) =0
\end{align*}
pointwise
in \(\Omega_T\).
Moreover, as proved in Lemma~\ref{Elam1}, $m\in \widehat\Aa$. Fix \(\delta>0\) and  \(\psi\in C_c^{\infty}(\Omega_T)\), and choose $(\eta,v)=( m,
\tilde u+\delta\psi )$ in (D2). Then, dividing by $\delta$, letting $\delta\to0$,
and using the arbitrariness of \(\psi\) and the identity  
\(\partial^\alpha_x  u=\partial^\alpha_x \tilde u\), we conclude that \((m,u)\) also satisfies 
\begin{align*}
&m_t - \sum_{i,j=1}^d(a_{ij}(x)m)_{x_ix_j} 
- \div\big(m D_p H(x,Du)\big)= 0
\end{align*}
pointwise
in \(\Omega_T\). Thus, \((m,u)\) is a classical solution to Problem~\ref{OP}. Finally, we observe that because 
 $g(m,h(\boldsymbol{m}))$ satisfies \eqref{stmono1} for $m_1, m_2\in L^2(\Omega_T)$ with $m_1\neq m_2$, the solution is unique. 
\end{proof}

%
\section{Properties of weak solutions to Problem~\ref{OP}}\label{prowsOP}
Next, we study properties of  weak solutions, $(m,u)$, to Problem~\ref{OP} in the sense of Definition~\ref{DOPWS1}. In particular, we show that $m$ and $u$ satisfy a transport equation and a Hamilton--Jacobi equation, respectively, in a weak sense.
Hereafter, we fix $\sigma\equiv0$ and $\xi\equiv0$, and we consider
the case in which  
\begin{equation}\label{exaHg}
H(x,p)=\frac{1}{2}|p|^2, \quad g(m,\theta) = m^{r} + \theta, \quad 
h(\boldsymbol{m})(x,t)=\big(\zeta\ast(\zeta\ast m)(\cdot, x)\big)(t),
\end{equation}
where $r>0$ and $\zeta\in C_c^\infty(\Tt^d)$ is such that $\zeta \geq0$, $\|\zeta\|_{L^2(\Tt^d)}=1$, and, for any $f$, $g \in L^1(\Tt^d)$, we have $\int_{\Tt^d}f(x)\, (\zeta\ast g)(x)\,\dx= \int_{\Tt^d}(\zeta \ast f)(x)\,g(x)\,\dx$.
Then, recalling Remark~\ref{esth1}, it can be checked that  Assumptions~\ref{Hconv}--\ref{Hmono} hold with \(\gamma=2\).
\begin{lem}\label{wlmJ}
Let \(H\), \(g\), and \(h\) be given by \eqref{exaHg},  
 let $(m_\epsilon,u_\epsilon)$ be the unique weak solution to Problem~\ref{MP} in the sense of Definition~\ref{DMPWS1} (with
\(\sigma=0\) and \(\xi=0\)),
and set \(q:=\frac{2(1+r)}{2+r}\).  Then, there exist $m\in L^{1+r}(\Omega_T)$ and $J\in L^{q}(\Omega_T;\Rr^d)$ such that $m_\epsilon\rightharpoonup m$ weakly in $L^{1+r}(\Omega_T)$ and $m_\epsilon Du_\epsilon \rightharpoonup J$ weakly in $L^{q}(\Omega_T;\Rr^d)$ as $\epsilon \to 0$, extracting a subsequence if necessary.
\end{lem}
\begin{proof}
Using  the fact that \(m_\epsilon g(m_\epsilon, h(\boldsymbol{m}_\epsilon))=m_\epsilon^{1+r}
 +m_\epsi h(\boldsymbol{m}_\epsilon) \geq m_\epsilon^{1+r}
   \) because \(m_\epsi h(\boldsymbol{m}_\epsilon)\geq 0\), Proposition~\ref{apriWS1}
and Corollary~\ref{apriDu1}
yield%
\begin{equation}\label{lem81a}
\int_0^T\!\!\!\int_{\Tt^d}m_\epsilon^{1+r}\,\dx\dt
\leq
\int_0^T\!\!\!\int_{\Tt^d}m_\epsilon g(m_\epsilon, h(\boldsymbol{m}_\epsilon))\,\dx\dt \leq C( 1 + \Vert Du_\epsi\Vert_{L^1(\Omega_T)}) \leq C,
\end{equation}
where $C$ is a positive constant independent of $\epsilon$. 

Next, we estimate \(m_\epsi Du_\epsi\) in \(L^{q}(\Omega_T;\Rr^d)\). We first
note that \(q = 1+ \frac{r}{2+r}>1
\) and \(\frac{q}{2(1+r)} + \frac{q}{2}
=1 \). Then, using Young's inequality,  \eqref{lem81a}, Proposition~\ref{apriWS1}, and Corollary~\ref{apriDu1},
  we get
\begin{equation}\label{lem81b}
\begin{aligned}
\int_0^T\!\!\!\int_{\Tt^d} m_\epsilon^{q}|Du_\epsilon|^{q}\,\dx\dt
 &=\int_0^T\!\!\!\int_{\Tt^d} m_\epsilon^{\frac{q}{2}}\Big(m_\epsilon^{\frac{1}{2}}|Du_\epsilon|\Big)^q\,\dx\dt\\
& \leq \frac{q}{2(1+r)}
\int_0^T\!\!\!\int_{\Tt^d} m_\epsilon^{1+r}\,\dx\dt
+\frac{q}{2}
\int_0^T\!\!\!\int_{\Tt^d} m_\epsi|Du_\epsilon|^2\,\dx\dt
\leq C,
\end{aligned}
\end{equation}
where $C$ is another positive constant independent of $\epsilon$. 

The conclusion follows from \eqref{lem81a} and \eqref{lem81b}
together with the fact that  
$L^{1+r}(\Omega_T)$ and   $L^q(\Omega_T;\Rr^d)$ are reflexive
Banach spaces. 
\end{proof}

\begin{pro}\label{Pwsm1}
Under the assumptions of Lemma~\ref{wlmJ},  let $m\in L^{1+r}(\Omega_T)$ and $J\in L^q(\Omega_T;\Rr^d)$ be given by Lemma~\ref{wlmJ} and
let \(\Bb_0\) be the set introduced in \eqref{DBb0}.  Then, for all $v\in \Bb_0$, $m$ satisfies
\begin{equation}\label{eq:weakFP}
-\int_{\Tt^d}v(0,x)m_0(x)\,\dx
-\int_0^T\!\!\!\int_{\Tt^d} \Big[
\big(v_t -\sum_{i,j=1}^da_{ij}(x)v_{x_ix_j}\big)m + J \cdot Dv \Big]\, \dx\dt =0.
\end{equation}
\end{pro}
\begin{proof}
Let $(m_\epsilon,u_\epsilon)$ be the unique weak solution to Problem~\ref{MP} in the sense of Definition~\ref{DMPWS1} and fix $v\in \Bb_0$; that is, $v\in H^{2k}(\Omega_T)$ and $v(x,T)=0$. Then, using (E3) and integration by parts,  we have
\begin{align*}
&-\int_{\Tt^d}v(0,x)m_0(x)\,\dx
-\int_0^T\!\!\!\int_{\Tt^d} \Big[ 
\Big(v_t-\sum_{i,j=1}^da_{ij}(x)v_{x_ix_j} \Big)m_\epsilon
+
m_\epsilon Du_\epsilon \cdot Dv
\Big]\, \dx\dt\\
&\quad=
-\epsilon\int_0^T\!\!\!\int_{\Tt^d}\Big( u_\epsi v +
\sum_{|\beta|=2k}\partial_{t,x}^{\beta}u_\epsilon \partial_{t,x}^{\beta}v
\Big)\,\dx\dt.
\end{align*}
By Corollary~\ref{aprisqrt} and H\"older's inequality, the right-hand side of the previous equality
converges to zero as \(\epsi\to0\). Hence, using the fact that $m_\epsilon\rightharpoonup
m$ weakly in $L^{1+r}(\Omega_T)$ and $m_\epsilon Du_\epsilon
\rightharpoonup J$ weakly in $L^{q}(\Omega_T;\Rr^d)$, as \(\epsi\to0\),
 together with the regularity of \((a_{ij})_{1\leq
i,j\leq d}\) and \(v\), letting \(\epsi\to0\) in the 
equality above yields \eqref{eq:weakFP}. 
\end{proof}%
\begin{remark}
The preceding proposition gives that $(m,J)$ is a weak solution to the following equation:
\begin{equation*}
\begin{cases}
m_t -\sum_{i,j=1}^d(a_{ij}(x)m)_{x_ix_j} - \div(J)=0 &\mbox{in }\Omega_T,\\
m(0,x)=m_0(x)  &\mbox{on }\Tt^d.
\end{cases}
\end{equation*}
\end{remark}

\begin{pro}\label{Pwsu1}
Let \(H\), \(g\), and \(h\) be given by \eqref{exaHg} with \(r\in(0,1]\).  
 Let $(m,u)$ be a weak solution to Problem~\ref{OP} in the sense of Definition~\ref{DOPWS1}, obtained as a sublimit of \(\{(m_\epsi,u_\epsi)\}_\epsi\) with  $(m_\epsilon,u_\epsilon)$  the unique weak solution to
Problem~\ref{MP} in the sense of Definition~\ref{DMPWS1} (with
\(\sigma=0\) and \(\xi=0\)). Then, for all $\varphi \in H^{2k}(\Omega_T)$ such that $\varphi\geq0$ and $\varphi(0,\cdot)=0$, we have 
\begin{align*}
&-\int_{\Tt^d}u_T(x)\varphi(T,x)\,\dx
+\int_0^T\!\!\!\int_{\Tt^d}
\Big(
u\varphi_t + \sum_{i,j=1}^du_{x_i}(a_{ij}\varphi)_{x_j}
\Big)\,\dx\dt\\
&\quad
+\int_0^T\!\!\!\int_{\Tt^d} \Big( \frac{1}{2}|Du|^2 - m^{r} - \zeta\ast(\zeta\ast m) - V\Big)\varphi\,\dx\dt
\leq 0.
\end{align*}
\end{pro}
\begin{proof}
Let \((m,u)\) and   $(m_\epsilon,u_\epsilon)\in
\mathcal{A}\times \mathcal{B}
$ be as stated, and fix $\varphi \in H^{2k}(\Omega_T)$ such that $\varphi\geq0$ and $\varphi(0,\cdot)=0$.
  Taking $w=m_\epsilon+\varphi\in\Aa$ in (E2) and integrating by parts, we obtain
\begin{equation}\label{eq:subsHJ}
\begin{aligned}
&-\int_{\Tt^d}u_T\varphi(T,x)\,\dx
+
\int_0^T\!\!\!\int_{\Tt^d}\Big(
 u_\epsilon\varphi_t + \sum_{i,j=1}^d{u_\epsilon}_{x_i}(a_{ij}\varphi)_{x_j}
 - V\varphi
\Big)\,\dx\dt +\int_0^T\!\!\!\int_{\Tt^d} \frac{1}{2}|Du_\epsilon|^2
\varphi\,\dx\dt   \\
&\quad- \int_0^T\!\!\!\int_{\Tt^d}m_\epsilon^{r}
\varphi\,\dx\dt - \int_0^T\!\!\!\int_{\Tt^d}\zeta\ast(\zeta\ast m_\epsilon)\varphi \,\dx\dt\leq
\epsilon\int_0^T\!\!\!\int_{\Tt^d}\Big(m_\epsilon\varphi +\sum_{|\beta|=2k} \partial_{t,x}^\beta m_\epsilon\partial_{t,x}^\beta \varphi\Big)\,\dx\dt.
\end{aligned}
\end{equation}
Next, 
we  pass \eqref{eq:subsHJ} to the limit as \(\epsi\to0\). First, we observe that  
 Corollary~\ref{aprisqrt} and H\"older's inequality yield
\begin{equation}
\label{eq:limHJ1}
\begin{aligned}
\lim_{\epsi\to0} \epsilon\int_0^T\!\!\!\int_{\Tt^d}\Big(m_\epsilon\varphi
+\sum_{|\beta|=2k} \partial_{t,x}^\beta
m_\epsilon\partial_{t,x}^\beta \varphi\Big)\,\dx\dt
= 0.
\end{aligned}
\end{equation}
Using the fact that \(u_\epsi \rightharpoonup
u
\) and \({u_\epsilon}_{x_i} \rightharpoonup
{u}_{x_i}\) weakly in \(L^2(\Omega_T)\)
for all \(i\in\{1,...,d\}\) (see Lemma~\ref{WCWS1})
together with the regularity of \((a_{ij})_{1\leq
i,j\leq d}\) and \(\varphi\), we conclude
that
\begin{equation}
\label{eq:limHJ2}
\begin{aligned}
&\lim_{\epsi\to0} \bigg[-\int_{\Tt^d}u_T\varphi(T,x)\,\dx
+
\int_0^T\!\!\!\int_{\Tt^d}\Big(
 u_\epsilon\varphi_t + \sum_{i,j=1}^d{u_\epsilon}_{x_i}(a_{ij}\varphi)_{x_j}
 - V\varphi
\Big)\,\dx\dt\bigg]\\
&\quad= -\int_{\Tt^d}u_T(x)\varphi(T,x)\,\dx
+\int_0^T\!\!\!\int_{\Tt^d}
\Big(
u\varphi_t + \sum_{i,j=1}^du_{x_i}(a_{ij}\varphi)_{x_j}
- V\varphi
\Big)\,\dx\dt.
\end{aligned}
\end{equation}
Similarly, the weak convergence \(m_\epsi
\rightharpoonup m\)  in \(L^1(\Omega_T)\)
and the symmetry of \(\zeta\) yield
\begin{equation}
\label{eq:limHJ3}
\begin{aligned}
\lim_{\epsi\to0}  \int_0^T\!\!\!\int_{\Tt^d}\zeta\ast(\zeta\ast
m_\epsilon)\,\varphi  \,\dx\dt
 &=
\lim_{\epsilon\to0}\int_0^T\!\!\!\int_{\Tt^d}
\zeta \ast(\zeta\ast \varphi)\,  m_\epsilon
\,\dx\dt
\\&=
\int_{\Tt^d}
\zeta \ast(\zeta\ast \varphi)\,  m
\,\dx\dt
=
\int_0^T\!\!\!\int_{\Tt^d}
\zeta\ast(\zeta\ast m)\,\varphi \,\dx\dt.
\end{aligned}
\end{equation}
Finally, recalling that \(\varphi
\geq 0\) and \(r\in(0,1]\), we have that for each \((x,t)\in\Omega_T\),
the maps \(p\in \Rr^d\mapsto
\frac{1}{2}|p|^2\varphi(x,t) \) and
\(m\in \Rr^+_0 \mapsto - m^r\varphi(x,t)\)
define convex functions over \(\Rr^d\)
and \(\Rr^+_0\), respectively.  Then,
the lower semicontinuous result  \cite[Theorem 6.54]{FoLe07} and the weak
convergence of \(\{m_\epsi\}_\epsi\)
and \(\{Du_\epsi\}_\epsi\) mentioned above
imply that
\begin{equation*}
\begin{aligned}
\int_0^T\!\!\!\int_{\Tt^d}
\frac{1}{2}|Du|^2
\varphi\,\dx\dt   - \int_0^T\!\!\!\int_{\Tt^d}m^{r}
\varphi\,\dx\dt&\leq \liminf_{\epsilon\to0}
\int_0^T\!\!\!\int_{\Tt^d}
\frac{1}{2}|Du_\epsilon|^2
\varphi\,\dx\dt +   \liminf_{\epsilon\to0}\bigg(\!- \int_0^T\!\!\!\int_{\Tt^d}m_\epsilon^{r}
\varphi\,\dx\dt\bigg) \\ & \leq
\liminf_{\epsilon\to0}\bigg(\int_0^T\!\!\!\int_{\Tt^d}
\frac{1}{2}|Du_\epsilon|^2
\varphi\,\dx\dt   - \int_0^T\!\!\!\int_{\Tt^d}m_\epsilon^{r}
\varphi\,\dx\dt\bigg).
\end{aligned}
\end{equation*}
This estimate, together with \eqref{eq:subsHJ}--\eqref{eq:limHJ3},
concludes the proof of 
 Proposition~\ref{Pwsu1}.
\end{proof}
\begin{remark}
Proposition~\ref{Pwsu1} still holds
if we replace the quadratic Hamiltonian,
\(H\), in \eqref{exaHg} by a Hamiltonian
satisfying Assumptions~\ref{Hconv}--\ref{Bderi}.
\end{remark}
\begin{remark}
For \(r\in(0,1]\), Proposition~\ref{Pwsu1} implies that $u$ is a subsolution to the following Hamilton--Jacobi equation:
\begin{equation*}
\begin{cases}
-u_t +\sum_{i,j=1}^da_{ij}(x)u_{x_ix_j}
+\frac{1}{2}|Du|^2 - m^{r}-\zeta\ast(\zeta\ast m)-V=0&\mbox{in } \Omega_T,\\
u(T,x)=u_T(x)&\mbox{on }\Tt^d.
\end{cases}
\end{equation*}
\end{remark}

%
%
\section{Final remarks}\label{Secfr}
In this section, we show how our methods
can be adapted to address other
 MFG models.
More precisely, we address the existence
of weak solutions to a MFG with congestion
and to a density constrained MFG.

%
\subsection{MFGs with congestion}
Here, we explain how the methods we
developed in the previous sections can
be used to prove the existence of weak
solutions to problems with congestion,
whose underlying Hamiltonian is singular at \(m=0\).
More concretely, we consider the following MFG with congestion.
\begin{problem}\label{fr1}
Let $T>0$ and $d\in \Nn$, and define $\Omega_T=(0,T)\times\Tt^d$. Let $X(\Omega_T)$ and $\Mm_{ac}(\Omega_T)$ be the spaces introduced in Problem~\ref{OP}. Fix $\epsilon \in (0,1)$ and $k \in \Nn$ such that $2k>\frac{d+1}{2} +3$. 
Assume that $a_{ij}\in C^2(\Tt^d)$ for $1\leq i,j\leq d$, $V \in L^{\infty}(\Omega_T)\cap C(\Omega_T)$, $g\in C^1(\Rr_0^+\times \Rr)$, $h:\Mm_{ac}(\Omega_T)\to X(\Omega_T)$ is a (possible nonlinear) operator, $m_0, u_T\in C^{4k}(\Tt^d)$, and $H\in C^2(\Tt^d\times\Rr^d\times \Rr^+)$ are such that, for $x\in \Tt^d$, $A(x)=(a_{ij}(x))$ is a symmetric positive semi-definite matrix, $m_0>0$, $\int_{\Tt^d}m_0(x)\, \dx=1$, and $m\mapsto g(m,h(\boldsymbol{m}))$ is monotone with respect to the $L^2$ inner product.
Find $(m,u)\in L^1(\Omega_T)\times L^\gamma((0,T);W^{1,\gamma}(\Tt^d))$ with $m>0$ solving
\begin{equation*}
\begin{cases}
u_t + \sum_{i,j=1}^da_{ij}(x)u_{x_ix_j} - H(x,Du,m) +g(m,h(\boldsymbol{m})) + V(t,x)=0 & {\rm in}\  \Omega_T, \\[1mm]
 m_t - \sum_{i,j=1}^d(a_{ij}(x)m)_{x_ix_j}-\div\big(mD_pH(x,Du,m)\big) = 0 & {\rm in} \ \Omega_T, \\[1mm]
 m(0,x)=m_{0}(x),\enspace u(T,x)=u_T(x)& \rm{on}\ \Tt^d. 
\end{cases}
\end{equation*} 
\end{problem}
Fix $\delta_0\in (0,1)$ such that \(m_0\geq \delta_0\) in \(\Tt^d\). Then, recalling $\Aa$ and $\widehat \Aa$ given by \eqref{DAa} and \eqref{DAawh} in Section~\ref{intro} respectively, we define 
\begin{align*}
\Aa_{\delta_0} := \big\{\eta\in \Aa\ |\ \eta \geq\delta_0\big\}, \ \ 
\widehat \Aa_{\delta_0}:=\big\{\eta\in \widehat \Aa\ |\ \eta \geq \delta_0\big\}.
\end{align*}
We introduce a notion of weak solutions
similar to that in Definition~\ref{DOPWS1}. %
\begin{definition}\label{dwsfr1}
A weak solution to Problem \ref{fr1} is a pair $(m,\tilde u) \in L^1(\Omega_T) \times L^\gamma((0,T);W^{1,\gamma}(\Tt^d))$ satisfying
\begin{flalign}
({\rm F1})\enspace  
&\,m \geq 0 \ \mbox{a.e~in }\Omega_T,  &\nonumber \\ 
({\rm F2})\enspace &  \left\langle
   \widehat F
  \begin{bmatrix}
      \eta \\
      v
  \end{bmatrix},
  \begin{bmatrix}
      \eta  \\
      v
  \end{bmatrix}
  -
  \begin{bmatrix}
      m  \\
      \tilde u
  \end{bmatrix}
\right\rangle \geq 0
\quad \mbox{for all}\ (\eta, v)\in \widehat\Aa_{\delta_0} \times \Bb,&\nonumber
\end{flalign}
where, for $(\eta, v) \in H^{2k}(\Omega_T) \times H^{2k}(\Omega_T)$ fixed, $\widehat F[\eta, v]: L^1(\Omega_T) \times L^1(\Omega_T) \to \Rr$ is the functional given by
\begin{equation}\label{DFfr1}
\begin{aligned}
\left\langle
   \widehat F
  \begin{bmatrix}
      \eta \\
      v
  \end{bmatrix},
  \begin{bmatrix}
      w_1 \\
      w_2
  \end{bmatrix}
\right\rangle
&:=
\int_0^T\!\!\!\int_{\Tt^d}\Big( v_t + \sum_{i,j=1}^da_{ij}(x)v_{x_ix_j} -H(x,Dv,\eta) + g(\eta, h(\boldsymbol{\eta})) + V(t,x)\Big)w_1 \, \dx\dt\\ 
&\quad+\int_0^T\!\!\!\int_{\Tt^d}\bigg[
\eta_t - \sum_{i,j=1}^d(a_{ij}(x)\eta)_{x_ix_j} - \div\Big(\eta D_pH(x,Dv,\eta)\Big) \bigg]w_2\, \dx\dt .
\end{aligned}
\end{equation}
\end{definition}

Instead of Assumptions~\ref{Hconv}--\ref{Bderi}, we suppose the next five assumptions
that, for instance, hold for
\begin{equation*}
H(x,p,m)=c(x)\frac{|p|^\gamma}{m^\tau},
\end{equation*}
where $c\in C^\infty(\Tt^d)$ is positive and $\tau\in (0,1)$.
%
\begin{hyp}\label{Hconvc}
For all $(x,m)\in \Tt^d\times \Rr^+$, $p\mapsto H(x,p,m)$ is convex in $\Rr^d$.
\end{hyp}
%
\begin{hyp}\label{Hcoer2}
There exists a constant, $C>0$, and $\tau\in (0,1)$ such that, for all $(x,p,m)\in \Tt^d\times \Rr^d\times \Rr^+$, we have
\begin{equation*}
-H(x,p,m) + D_pH(x,p,m)\cdot p
\geq
\frac{1}{C}\frac{|p|^\gamma}{m^\tau}-C.
\end{equation*}
\end{hyp}
%
\begin{hyp}\label{Hbdd2}
Let $\tau$ be as in Assumption~\ref{Hcoer2}. There exists a constant, $C>0$, such that, for all $(x,p,m)\in \Tt^d\times\Rr^d\times \Rr^+$,
we have%
\begin{equation*}
\begin{aligned}
H(x,p,m)
\geq 
\frac{1}{C} \frac{|p|^\gamma}{m^\tau}-C.
\end{aligned}
\end{equation*}
\end{hyp}
%
\begin{hyp}\label{Bderi2}
Let $\tau$ be as in Assumption~\ref{Hcoer2}. There exists a constant, $C>0$, such that, for all $(x,p,m)\in \Tt^d\times\Rr^d\times \Rr^+$,
we have
\begin{equation*}
|D_pH(x,p,m)|\leq C\frac{|p|^{\gamma-1}}{m^\tau}+C .
\end{equation*}
\end{hyp}
Moreover, in place of and in analogy
with Assumption~\ref{Hmono},  we assume
the following  monotonicity condition on $\widehat F$.  
\begin{hyp}\label{Hmono2}
The functional $\widehat F$ introduced in Definition~\ref{dwsfr1} is monotone with respect to the $L^2\times L^2$-inner product; that is,
for all $(\eta_1, v_1)$, $(\eta_2, v_2) \in \Aa \times \Bb$ with $\eta_1, \eta_2 >0$, $\widehat F$ satisfies
\begin{equation*}
\left\langle
   \widehat F
  \begin{bmatrix}
      \eta_1  \\
      v_1 
  \end{bmatrix}
  -
   \widehat F
  \begin{bmatrix}
      \eta_2  \\
      v_2 
  \end{bmatrix},
  \begin{bmatrix}
      \eta_1  \\
      v_1
  \end{bmatrix}
  -
  \begin{bmatrix}
      \eta_2  \\
      v_2
  \end{bmatrix}
\right\rangle
\geq 0.
\end{equation*}
\end{hyp}
%
%

Because $m$ given by our previous construction is only nonnegative,
to address congestion we consider instead an approximation $m_\epsilon$ satisfying $m_\epsilon \geq \epsilon$ for some $\epsilon \in (0,1)$.
Also, we set a test function space for $m$, $\widehat\Aa_{\delta_0}$, such that $\eta \in \widehat\Aa_{\delta_0}$ 
satisfies $\eta\geq \delta_0>0$ in
 $\Omega_T$. By Morrey's theorem, for all $(\eta, v)\in \widehat\Aa_{\delta_0}\times \Bb$, we have $\eta, v \in C^{1,l}(\Omega_T)$ for some $l\in (0,1)$, and thus, we obtain $\frac{|Dv|^\gamma}{\eta^\tau}\in C(\overline\Omega_T)$. Therefore, since test function spaces have enough regularity, we use a proof similar to the one in Theorem~\ref{TOP} given in Section~\ref{PfMT2} and obtain the following theorem. 
%
\begin{teo}\label{frtm1}
Consider Problem~\ref{fr1} and suppose 
that Assumptions \ref{hhyp}--\ref{gwc} and \ref{Hconvc}--\ref{Hmono2} hold. Then, there exists a weak solution $(m,\tilde u)\in L^1(\Omega_T) \times L^\gamma((0,T);W^{1,\gamma}(\Tt^d))$ to Problem \ref{fr1} in the sense of Definition~\ref{dwsfr1}.
\end{teo}
\begin{proof}[Proof sketch]
First, for \(\epsi\in(0,\delta_0)\), we introduce the following regularized problem
\begin{equation}\label{rpMFGfr1}
\begin{cases}
\displaystyle u_t + \sum_{i,j=1}^da_{ij}(x)u_{x_ix_j} - H(x,Du,m) 
+ g(m,h(\boldsymbol{m})) 
+ V(t,x)
+ 
\epsi\sum_{|\beta|\in\{0,2k\}} \partial^{2\beta}_{t,x}(m+\sigma) 
=0 & \hbox{in } \Omega_T,\\
\displaystyle m_t - \sum_{i,j=1}^d\big(a_{ij}(x)(m+\sigma)\big)_{x_ix_j} 
- 
\div\big((m + \sigma)D_pH(x,Du,m)\big)
+\epsi\sum_{|\beta|\in\{0,2k\}} \partial^{2\beta}_{t,x}
(u+\xi)=0& \hbox{in } \Omega_T,\\
m(0,x)=m_{0}(x),\enspace u(T,x)=u_T(x) & {\rm on}\ \Tt^d, 
\end{cases}
\end{equation}
where $\sigma\in C^{4k}(\overline{\Omega}_T)$
is nonnegative and $\xi\in C^{4k}(\overline{\Omega}_T)$.  

Second, define
\begin{equation*}
\Aa_\epsilon:=\left\{
m \in H^{2k}(\Omega_T)\ |\
m(0, x)=m_0(x),\ m\geq \epsilon
\right\}
\end{equation*}
and 
we consider a notion of weak solution
to this regularized problem similar to the one in Definition~\ref{DMPWS1}  by replacing $\Aa$  by $\Aa_\epsilon$ for each $\epsilon$.

Arguing as in Proposition~\ref{apriWS1}, we have
\begin{equation}\label{Eq9.1}
\begin{split}
&\int_0^T\!\!\!\int_{\Tt^d} \Big(
mg(m,h(\boldsymbol{m}))
+
\frac{1}{C} (m + \sigma)\frac{|Du|^\gamma}{m^\tau} 
+
\frac{1}{C} m_{0} \frac{|Du|^\gamma}{m^\tau} \Big)\, \dx\dt\\
&\quad+
{\epsilon} \int_0^T\!\!\!\int_{\Tt^d} \
\Big(
m^2+u^2+\sum_{|\beta|=2k}(\partial_{t,x}^\beta m)^2
+ \sum_{|\beta|\leq 2k}(\partial_{t,x}^\beta u)^2
 \Big) \, \dx\dt
\leq C (1+\|Du\|_{L^1(\Omega_T)}),
\end{split}
\end{equation}
where $C$ is independent of $\epsilon$.
Moreover, as in the proof of Proposition~4.2 in \cite{FG2} with $0<\epsilon\leq1$, because $\tau\in (0,1)$, we have
\begin{equation}\label{Eq9.2}
\begin{split}
&\int_0^T\!\!\!\int_{\Tt^d} \Big(
m\frac{|Du|^\gamma}{m^\tau} + \frac{|Du|^\gamma}{m^\tau} \Big)\, \dx\dt \geq
\int\!\!\!\int_{\{m\geq 1\}}
m\frac{|Du|^\gamma}{(m+1)^\tau}\,\dx\dt 
+ 
\int\!\!\!\int_{\{m< 1\}}
\frac{|Du|^\gamma}{(m+1)^\tau}\, \dx\dt\\
\geq&
\int\!\!\!\int_{\{m\geq1\}}
\frac{m^{1-\tau}}{2^\tau}|Du|^\gamma\,\dx\dt
+
\int\!\!\!\int_{\{m<1\}}
\frac{1}{2^\tau}|Du|^\gamma\,\dx\dt
\geq
\frac{1}{2^\tau}\|Du\|_{L^\gamma(\Omega_T)}^\gamma.
\end{split}
\end{equation}
Then, since $m_0$ is strictly positive in $\Omega_T$,
combining \eqref{Eq9.1} with 
\eqref{Eq9.2} and  Young's inequality,
we conclude that  $\|Du\|_{L^\gamma(\Omega_T)}\leq C $, where $C$ is independent of $\epsilon$. Also, Corollary~\ref{aprisqrt} holds. Thus, the results in Section~\ref{Prows} hold.

Next, define
\begin{equation*}
\Aa_{\epsilon, 0}:=
\left\{
m\in H^{2k}(\Omega_T)\ | \ 
m(0,x)=0,\ m+m_0\geq \epsilon
\right\}
\end{equation*}
and recall $I_{(m, u)}$ given by \eqref{defImu}. Fix $(m_1, u_1)\in H^{2k-2}(\Omega_T)\times H^{2k-1}(\Omega_T)$ with $m_1 + m_0 \geq \epsilon$ and set $I_1=I_{(m_1,u_1)}$.
Then, we consider the following variational problem: find $m \in \Aa_{\epsilon,0}$ such that
\begin{equation}\label{fr1vp}
I_1[m]
=
\inf_{w\in \Aa_{\epsilon, 0}} I_1[w].
\end{equation}
By the same proof as in Proposition~\ref{EMVP1}, there exists a unique minimizer $m\in \Aa_{\epsilon, 0}$ satisfying \eqref{fr1vp}. 
Therefore, for each $\epsilon \in (0,1)$, we can apply the results of Sections~\ref{VP}--\ref{PfMT1}
(with the obvious modifications) to prove the existence and uniqueness of weak solutions to \eqref{rpMFGfr1} in the sense of Definition~\ref{DMPWS1} with $\Aa$ replaced by $\Aa_{\epsilon}$.
Then, because $\widehat \Aa_{\delta_0}\subset \Aa_\epsilon$ (as
\(0<\epsi<\delta_0\)), we obtain a unique weak solution $(m,u)$ satisfying (E1), (E2), and (E3) in Definition~\ref{DMPWS1} for all $w\in \widehat \Aa_{\delta_0}$ and $v\in \Bb_0$.
Since Proposition~\ref{apriWS1} and Corollary~\ref{apriDu1} hold and $m_\epsilon \geq 0$, applying Poincar\'{e}--Wirtinger inequality and Assumption~\ref{gwc}, Lemma~\ref{WCWS1} follows and thus Lemmas~\ref{PAm} and \ref{Equtu} also hold.

Set $\sigma\equiv 0$ and, for $(\eta, v)\in \Aa_{\delta_0}\times \Bb$, define $\widehat F_\epsilon$ by 
\begin{equation}\label{fr1eq1}
\begin{split}
\left\langle
  {\widehat F}_{\epsilon}
  \begin{bmatrix}
      \eta  \\
      v 
  \end{bmatrix},
      \begin{bmatrix}
          w_1  \\
          w_2
      \end{bmatrix}
\right\rangle
=
\left\langle
  {\widehat F}
  \begin{bmatrix}
      \eta  \\
      v 
  \end{bmatrix},
      \begin{bmatrix}
          w_1  \\
          w_2
      \end{bmatrix}
\right\rangle
&+
\epsilon\int_0^T\!\!\!\int_{\Tt^d}
\Big(
\eta w_1 + \sum_{|\beta|=2k}\partial_{t,x}^\beta \eta \partial_{t,x}^\beta w_1
\Big)
\,\dx\dt\\
&+
\epsilon\int_0^T\!\!\!\int_{\Tt^d}
\Big(
v w_2 + \sum_{|\beta|=2k}\partial_{t,x}^\beta v \partial_{t,x}^\beta w_2
\Big)
\,\dx\dt,
\end{split}
\end{equation}
where $\widehat F$ is given by \eqref{DFfr1}.
Let $(m_\epsilon, u_\epsilon)$ be the unique weak solution to \eqref{rpMFGfr1} given by (the analogue of) Theorem~\ref{MT}. Then, for $(\eta, v) \in \widehat \Aa_{\delta_0}\times \Bb$, we get
\begin{equation}\label{fr1eq2}
\left\langle
{\widehat F}_{\epsilon}
  \begin{bmatrix}
      m_\epsilon  \\
      u_\epsilon 
  \end{bmatrix},
      \begin{bmatrix}
          \eta  \\
          v
      \end{bmatrix}
      -
      \begin{bmatrix}
          m_\epsilon  \\
          u_\epsilon
      \end{bmatrix}
\right\rangle
\geq 0.
\end{equation}
Fix $(\eta,v)\in \widehat\Aa_{\delta_0}\times \Bb$.
Since $\eta\geq \delta_0>0$ and  $m_\epsilon\geq \epsilon$, from Assumption~\ref{Hmono2}, we have
\begin{equation}\label{fr1eq3}
\left\langle
  {\widehat F}
  \begin{bmatrix}
      \eta  \\
      v 
  \end{bmatrix}
  -
    {\widehat F}
  \begin{bmatrix}
      m_\epsilon  \\
      u_\epsilon 
  \end{bmatrix},
      \begin{bmatrix}
          \eta  \\
          v
      \end{bmatrix}
      -
      \begin{bmatrix}
          m_\epsilon  \\
          u_\epsilon
      \end{bmatrix}
\right\rangle
\geq 0.
\end{equation}
Therefore, by \eqref{fr1eq1}, \eqref{fr1eq2}, and \eqref{fr1eq3}, we obtain
\begin{equation*}
0\leq \left\langle
  {\widehat F}
  \begin{bmatrix}
      \eta  \\
      v 
  \end{bmatrix}
  -
    {\widehat F}
  \begin{bmatrix}
      m_\epsilon  \\
      u_\epsilon 
  \end{bmatrix},
      \begin{bmatrix}
          \eta  \\
          v
      \end{bmatrix}
      -
      \begin{bmatrix}
          m_\epsilon  \\
          u_\epsilon
      \end{bmatrix}
\right\rangle
\leq
\left\langle
  {\widehat F}_{\epsilon}
  \begin{bmatrix}
      \eta  \\
      v 
  \end{bmatrix}
  -
    {\widehat F}_{\epsilon}
  \begin{bmatrix}
      m_\epsilon  \\
      u_\epsilon 
  \end{bmatrix},
      \begin{bmatrix}
          \eta  \\
          v
      \end{bmatrix}
      -
      \begin{bmatrix}
          m_\epsilon  \\
          u_\epsilon
      \end{bmatrix}
\right\rangle
\leq
\left\langle
  {\widehat F}_{\epsilon}
      \begin{bmatrix}
         \eta  \\
         v 
      \end{bmatrix},
      \begin{bmatrix}
          \eta  \\
          v
      \end{bmatrix}
      -
      \begin{bmatrix}
          m_\epsilon  \\
          u_\epsilon
      \end{bmatrix}
\right\rangle.
\end{equation*}
Recalling $c_\epsilon$ given by \eqref{vanish1}, we get
\begin{equation*}
\left\langle
  {\widehat F}_{\epsilon}
      \begin{bmatrix}
         \eta  \\
         v 
      \end{bmatrix},
      \begin{bmatrix}
          \eta  \\
          v
      \end{bmatrix}
      -
      \begin{bmatrix}
          m_\epsilon  \\
          u_\epsilon
      \end{bmatrix}
\right\rangle
=
\left\langle
  {\widehat F}
      \begin{bmatrix}
         \eta  \\
         v 
      \end{bmatrix},
      \begin{bmatrix}
          \eta  \\
          v
      \end{bmatrix}
      -
      \begin{bmatrix}
          m_\epsilon  \\
          u_\epsilon
      \end{bmatrix}
\right\rangle
+ c_\epsilon.
\end{equation*}
Combining H\"older's inequality with Corollary~\ref{aprisqrt}, we have $\lim_{\epsilon\to 0}c_\epsilon=0$. 
Therefore, applying (the analogue of) Lemmas~\ref{WCWS1} and
\ref{Equtu}, for $(\eta,v)\in \widehat \Aa_{\delta_0}\times \Bb$, we have
\begin{equation*}
0
\leq 
\lim_{\epsilon\to0}
\left\langle
  {\widehat F}
      \begin{bmatrix}
         \eta  \\
         v 
      \end{bmatrix},
      \begin{bmatrix}
          \eta  \\
          v
      \end{bmatrix}
      -
      \begin{bmatrix}
          m_\epsilon  \\
          u_\epsilon
      \end{bmatrix}
\right\rangle
+\lim_{\epsilon\to0} c_\epsilon
=
\lim_{\epsilon\to0}
\left\langle
  {\widehat F}
      \begin{bmatrix}
         \eta  \\
         v 
      \end{bmatrix},
      \begin{bmatrix}
          \eta  \\
          v
      \end{bmatrix}
      -
      \begin{bmatrix}
          m_\epsilon  \\
          \tilde u_\epsilon
      \end{bmatrix}
\right\rangle
=
\left\langle
  {\widehat F}
      \begin{bmatrix}
         \eta  \\
         v 
      \end{bmatrix},
      \begin{bmatrix}
          \eta  \\
          v
      \end{bmatrix}
      -
      \begin{bmatrix}
          m  \\
          u
      \end{bmatrix}
\right\rangle,
\end{equation*}
from which Theorem~\ref{frtm1} follows.
\end{proof}

%
\subsection{Density constraints}
Finally, we show how to apply our methods to prove the existence of weak solutions to MFGs with density constraints.
We consider the following MFGs with a density constraint. 
%
\begin{problem}\label{fr2}
Let $T>0$ and $d\in \Nn$, and define $\Omega_T=(0,T)\times\Tt^d$. Let $X(\Omega_T)$ and $\Mm_{ac}(\Omega_T)$ be the spaces introduced in Problem~\ref{OP}. Fix $M>1$, $\epsilon \in (0,1)$, and $k \in \Nn$ such that $2k>\frac{d+1}{2} +3$. 
Assume that $a_{ij}\in C^2(\Tt^d)$ for $1\leq i,j\leq d$, $V \in L^{\infty}(\Omega_T)\cap C(\Omega_T)$, $g\in C^1(\Rr_0^+\times \Rr)$, $h:\Mm_{ac}(\Omega_T)\to X(\Omega_T)$ is a (possible nonlinear) operator, $m_0, u_T\in C^{4k}(\Tt^d)$, and $H\in C^2(\Tt^d\times\Rr^d)$ are such that, for $x\in \Tt^d$, $A(x)=(a_{ij})$ is a symmetric positive semi-definite matrix, $0<m_0\leq M$, $\int_{\Tt^d}m_0(x)\, \dx=1$, and $m\mapsto g(m,h(\boldsymbol{m}))$ is monotone with
respect to the $L^2$-inner product.
Find $(m,u)\in L^1(\Omega_T)\times L^\gamma((0,T);W^{1,\gamma}(\Tt^d))$ satisfying $0\leq m\leq M$ and
\begin{equation*}
\begin{cases}
u_t + \sum_{i,j=1}^da_{ij}(x)u_{x_ix_j} - H(x,Du) +g(m,h(\boldsymbol{m})) + V(t,x)=0 & {\rm in}\  \Omega_T, \\[1mm]
 m_t - \sum_{i,j=1}^d(a_{ij}(x)m)_{x_ix_j}-\div\big(mD_pH(x,Du)\big) = 0 & {\rm in} \ \Omega_T, \\[1mm]
 m(0,x)=m_{0}(x),\enspace u(T,x)=u_T(x)& \rm{on}\ \Tt^d. 
\end{cases}
\end{equation*} 
\end{problem}
For density constraints, we define new function spaces for $m$ with the constraints. More precisely, we define the set 
\(
\widehat\Aa_1:=\Big\{m\in \widehat\Aa\ \big|\ 0\leq m\leq M\Big\},
\)
where $\widehat\Aa$ is given by \eqref{DAawh}.

We introduce a notion of weak solutions similar to that in Definition~\ref{DOPWS1}.
%
\begin{definition}\label{dwsfr2}
A weak solution to Problem \ref{fr2} is a pair $(m,\tilde u) \in L^1(\Omega_T) \times L^\gamma((0,T);W^{1,\gamma}(\Tt^d))$ satisfying
\begin{flalign}
({\rm G1})\enspace  &\,
0\leq m\leq M \ \mbox{a.e.~in }\Omega_T,  &\nonumber \\ 
({\rm G2})\enspace &  \left\langle
      F
  \begin{bmatrix}
      \eta \\
      v
  \end{bmatrix},
  \begin{bmatrix}
      \eta  \\
      v
  \end{bmatrix}
  -
  \begin{bmatrix}
      m  \\
      \tilde u
  \end{bmatrix}
\right\rangle \geq 0
\quad \mbox{for all}\ (\eta, v)\in \widehat\Aa_1 \times \Bb,&\nonumber
\end{flalign}
where, for $(\eta, v) \in H^{2k}(\Omega_T) \times H^{2k}(\Omega_T)$ fixed, $F[\eta, v]: L^1(\Omega_T) \times L^1(\Omega_T) \to \Rr$ is the functional given by \eqref{DRF2}.
\end{definition}

Under Assumptions~\ref{Hconv}--\ref{Hmono}, we obtain the following result.
%
\begin{teo}\label{frtm2}
Consider Problem~\ref{fr2} and suppose 
that Assumptions \ref{Hconv}--\ref{Hmono} hold. Then, there exists a weak solution $(m,\tilde u)\in L^1(\Omega_T) \times L^\gamma((0,T);W^{1,\gamma}(\Tt^d))$ to Problem \ref{fr2} in the sense of Definition~\ref{dwsfr2}.
\end{teo}
%
\begin{proof}[Proof sketch]
Define
\(\Aa_1:=\big\{m\in \Aa\ \big|\ 0\leq m\leq M\big\}\) 
and 
\(\Aa_{1,0}:=\big\{m\in \Aa_0\ \big|\ 0\leq m +m_0\leq M\big\}\),
where $\Aa$ and $\Aa_0$ are given by \eqref{DAa} and \eqref{DAa0} respectively.

First, we consider the same regularized problem given in Problem~\ref{MP}. We use the notion of weak solution, $(m_\epsilon, u_\epsilon)$, to the regularized problem given by Definition~\ref{DMPWS1} with $\Aa$ replaced by $\Aa_1$.
Since $u_T-u_\epsilon\in \Bb_0$ and $m_0\in \Aa_1$, Proposition~\ref{apriWS1} holds and thus Corollaries~\ref{apriDu1}--\ref{aprisqrt} also follow.

Fix $(m_1, u_1)\in H^{2k-2}(\Omega_T) \times H^{2k-1}(\Omega_T) $ with $0\leq m_1+m_0\leq M$. Recall $I_{(m,u)}$ defined by \eqref{defImu} and $I_1=I_{(m_1, u_1)}$.
Then, as the same argument in Section~\ref{VP}, we consider the following variational problem:
\begin{equation}\label{fr2vp1}
I_1[m]=\inf_{w\in \Aa_{1,0}}I_1[w].
\end{equation}
Let $ \{w_n\}_{n=1}^\infty  \subset  \Aa_{1,0}$ be a minimizing sequence for \eqref{fr2vp1}. As in the proof of Proposition~\ref{EMVP1}, $ \{w_n\}_{n=1}^\infty $ is bounded in $H^{2k}(\Omega_T)$ and $w_n \rightharpoonup m$ weakly in $H^{2k}(\Omega_T) $ for some $m\in H^{2k}(\Omega_T) $, extracting a subsequence if necessary. Moreover, by Morrey's theorem and Rellich--Kondrachov theorem, $w_n \to m$ in $C^{2,l}(\overline\Omega_T)$ for some $l\in(0,1)$. Since $0 \leq w_n +m_0 \leq M$ and $w_n(0,x)=0$, we have $0\leq m+m_0\leq M$ and $m(0,x)=0$. Hence, $m\in \Aa_{1,0}$. Thus, from the proof of Proposition~\ref{EMVP1}, there exists a unique minimizer $m\in\Aa_{1,0}$ satisfying \eqref{fr2vp1}.
Therefore, the results of Sections~\ref{VP}--\ref{BF} follow. 

Recalling the mapping $A$ given by \eqref{OpeA}, define $\widetilde \Aa_{1,0}=\{w\in H^{2k-2}(\Omega_T)\ |\ w(0,x)=0,\ 0 \leq w + m_0 \leq M \}$ and consider $A: \widetilde \Aa_{1,0}\times \widetilde\Bb \to \widetilde \Aa_{1,0}\times \widetilde\Bb $ defined, for $(m_1,u_1)\in \widetilde\Aa_{1,0}\times \widetilde\Bb_0$, by
\begin{equation*}
       A
    \begin{bmatrix}
      m_1  \\
      u_1        
    \end{bmatrix}
  :=
    \begin{bmatrix}
   m_1^\ast \\
   u_1^\ast
    \end{bmatrix},
\end{equation*}
where $m_1^\ast \in \Aa_{1,0}$ is the unique minimizer to \eqref{fr2vp1} and $u_1^\ast \in \Bb_0$ is the unique solution to \eqref{BL1}. Because $\widetilde \Aa_{1,0}$ is convex and closed, the results in Section~\ref{PfMT1} (with the obvious modifications) hold, and thus, we prove Theorem~\ref{MT}.

Let $(m_\epsilon, u_\epsilon) $ be the unique weak solution to Problem~\ref{MP} and recall $\tilde u_\epsilon$ given by \eqref{Defutild}.
Then, as in the same proof of Lemma~\ref{WCWS1}, there exists $(m,\tilde u)\in L^1(\Omega_T)\times L^\gamma((0,T); W^{1,\gamma}(\Tt^d))$ such that $(m_\epsilon, u_\epsilon)$ converges to $(m,\tilde u)$ weakly in $L^1(\Omega_T)\times L^\gamma((0,T);W^{1\gamma}(\Tt^d))$.
Since $0\leq m_\epsilon\leq M$ a.e.~in $\Omega_T$, we have $0\leq m\leq M$ a.e.~in $\Omega_T$.
Also, Lemmas~\ref{PAm}--\ref{MORF} and \ref{Equtu} follow. Therefore, arguing as in the proof of Theorem~\ref{TOP}, we 
establish the existence of a solution of Problem \ref{fr2}. 
\end{proof}

\bibliographystyle{plain}

\def\cprime{$'$}

\end{document}